\documentclass[12pt,letterpaper]{amsart}
\usepackage[utf8]{inputenc}

\usepackage[english]{babel}
\usepackage{float}

\usepackage[normalem]{ulem}
\usepackage{cancel}

\usepackage{verbatim}

\usepackage{latexsym}    
\usepackage{amssymb}   
\usepackage{amsmath}    
\usepackage{amsbsy}
\usepackage{amsthm}
\usepackage{amsgen}
\usepackage{amsfonts}
\usepackage{array}
\usepackage{lmodern}

\usepackage{graphicx}
\usepackage{caption}
\usepackage[labelformat=simple]{subcaption}

\newsavebox{\largestimage}

\usepackage{tikz}
\usepackage{tikz-cd}
\usetikzlibrary{decorations.pathmorphing}
\usetikzlibrary{calc, decorations.markings}

\usepackage{xcolor}

\usepackage{verbatim}

\usepackage{url}
\usepackage[shortlabels]{enumitem}

\usepackage{lineno}

\DeclareMathOperator{\expo}{exp}

\DeclareMathOperator{\supp}{supp}

\newcommand{\Cli}{C\ell}
\newcommand{\D}{\mathbb{D}} 
\newcommand{\End}{\mathrm{End}}
\newcommand{\fol}{\mathcal{F}} 
\newcommand{\Hol}{\mathrm{Hol}} 
\newcommand{\N}{\mathbb{N}}
\newcommand{\prin}{\mathrm{prin}} 
\newcommand{\R}{\mathbb{R}} 
\newcommand{\Sp}{\mathbb{S}} 
\newcommand{\Sym}{\mathrm{Sym}^2}

\newcommand{\tr}{\mathrm{tr}}

\newtheorem{theorem}{Theorem}  
\newtheorem{corollary}[theorem]{Corollary}  

	\newtheorem{thm}{Theorem}[section]
	
	\newtheorem{lemma}[thm]{Lemma}

	\newtheorem{proposition}[thm]{Proposition}
	\newtheorem*{question}{Question}
	
\theoremstyle{definition}	
	\newtheorem{remark}[thm]{Remark}
	\newtheorem{definition}[thm]{Definition}

 \newtheoremstyle{TheoremNum}
        {\topsep}{\topsep}              
        {\itshape}                      
        {}                              
        {\bfseries}                     
        {.}                             
        { }                             
        {\thmname{#1}\thmnote{ \bfseries #3}}
    \theoremstyle{TheoremNum}

\usepackage{hyperref}
\hypersetup{
    colorlinks=true,
    linkcolor=blue,
    citecolor=blue,
    urlcolor=blue,
    pdfauthor={Diego Corro, Juan Carlos Fernandez, Raquel Perales},
    pdftitle={Yamabe problem in the presence of singular Riemannian  Foliations}
}
\usepackage[reftex]{theoremref}

\title[Yamabe problem on Foliations]{Yamabe problem in the presence of singular Riemannian  Foliations}


\author[D.~Corro]{Diego Corro$^{\ast}$}
\address[D.~CORRO]{Mathematics Institute of the National Autonomous University of Mexico (UNAM), Oaxaca, Mexico.}
\curraddr{Mathematisches Institut, Universität zu Köln, Köln, Deutschland.}
\email{\href{mailto:diego.corro.math@gmail.com}{diego.corro.math@gmail.com}}
\thanks{$^\ast$Supported by DGAPA-Fellowship associated to the  Mathematics Institute of UNAM, campus Oaxaca, and by DFG-Eigenestelle Fellowship CO 2359/1-1}

\author[J.~C.~Fernández]{Juan Carlos Fernandez$^{\ast\ast}$}
\address[J.~C.~FERNÁNDEZ]{Facultad de Ciencias, Universidad Na\-cio\-nal Au\-tó\-no\-ma de México (UNAM), Ciudad de Mé\-xi\-co, Mé\-xi\-co.}
\email{\href{mailto:jcfmor@ciencias.unam.mx}{jcfmor@ciencias.unam.mx}}
\thanks{$^{\ast\ast}$Partially supported by Professor Christina Sormani's NSF Reaserch Grant DMS-1612049}

\author[R.~Perales]{Raquel Perales}
\address[R.~PERALES]{Investigadora por M\'exico at the Mathematics Institute of the National Autonomous University of Mexico (UNAM), Oaxaca, Mexico.}
\email{\href{mailto:raquel.perales@im.unam.mx}{raquel.perales@im.unam.mx}}


\begin{document}

\overfullrule=2mm

\maketitle
\begin{abstract}
    Using variational methods together with symmetries given by singular Riemannian foliations with positive dimensional leaves, we prove the existence of an infinite number of sign-changing solutions to Yamabe type problems, which are constant along the leaves of the foliation, and one positive solution of minimal energy among any other solution with these symmetries. In particular, we find sign-changing solutions to the Yamabe problem on the round sphere with new qualitative behavior when compared to previous results, that is, these solutions are constant along the leaves of a singular Riemannian foliation which is not induced neither by a group action nor by an isoparametric function. To prove the existence of these solutions, we prove a Sobolev embedding theorem for general singular Riemannian foliations, and a Principle of Symmetric Criticality for the associated energy functional to a Yamabe type problem.
\end{abstract}

\section{Introduction}

The Yamabe problem, stated by Yamabe in \cite{Yamabe1960}, asks if for a given closed Riemannian manifold $(M,g)$ there exists a conformal Riemannian metric $\hat{g} = fg$, for some smooth positive function $f\colon M\to \R$, such that the scalar curvature of $\hat{g}$ is constant. This problem can be written in terms of a PDE, and has been completely solved in the positive by the combined work of Yamabe \cite{Yamabe1960}, Trudinger \cite{Trudinger1968}, Aubin \cite{Aubin1976}, and Schoen \cite{Schoen1984}. Solutions to this problem are not necessarily unique, although for manifolds not conformally equivalent to the round sphere with dimension less than or equal to $24$, the set of solutions is compact (see \cite{KhuriMarquesSchoen2009}). In contrast, for higher dimensions the set of solutions is not compact (see \cite{Brendle2008,BrendleMarques2009}). For an extensive exposition of the Yamabe problem, the interested reader may consult \cite{Aubin1998,LeeParker1987}.

In the past years, sign-changing solutions to the Yamabe equation have been studied. These are functions $u$ which are sign changing and satisfy the Yamabe equation. For $k,n\geq 2$, Ding \cite{Ding1986} established the existence of infinitely many sign-changing solutions for the round $(k+n)$-sphere, using the fact that there is a (linear) $\mathrm{O}(k)\times \mathrm{O}(n)$ action by isometries. Moreover, Ammann and Humbert proved in \cite{AmmannHumbert2006} that in dimension at least $11$ for a closed Riemannian manifold with positive Yamabe invariant and not locally conformally flat, there exists a minimal energy sign-changing solution to the Yamabe equation. 

The study of the existence of sign-changing solutions has been recently carried out by several authors using very different techniques \cite{AmmannHumbert2006,ClappPacella2008,ClappFernandez2017,delPinoMussoPacardPistoia2011,FernandezPetean(2020),Henry2019,Vetois2007}. One of the approaches considered has been to find equivariant solutions with respect to a given compact Lie group action by isometries with positive dimensional orbits. The approach of finding equivariant solutions to the Yamabe equation has also led to an equivariant solution to the Kazdan-Warner problem in \cite{CavenaghiDoOSperanca2021}.

Recently, singular Riemannian foliations have been considered as a notion of symmetry for Riemannian manifolds in the context of manifolds with nonnegative sectional curvature, since they are a natural extension to the concepts of group actions and Riemannian submersions (see for example \cite{Alexandrino,CorroMoreno2020,Galaz-Garcia2015,GeRadeschi2013}). Moreover, several results which hold for group actions that depend only on the  geometry that is transverse  to the orbits can be extended to the setting of singular Riemannian foliations.

In light of this, we study the existence of sign-changing solutions for a family of elliptic partial differential equations related to the  Yamabe problem in the presence of singular Riemannian foliations. 

Namely, let $(M,g)$ be a closed Riemannian manifold of dimension $m\geq 3$ and consider the following Yamabe type problem:
\begin{equation}\label{Main:Yamabe type}
	-\Delta_g u + b u = c \vert u\vert^{p-2}u\tag{Y} \quad\text{on }\ M,
\end{equation}
where $\Delta_g=\text{div}_g\text{grad}_g$ is the Laplace-Beltrami operator, $p>2$, $b,c\in\mathcal{C}^\infty(M)$ with $c>0$. We will assume that the operator $-\Delta_g + b$ is \emph{coercive} in the Sobolev space $H_g^1(M)$, meaning, that there exists $\mu>0$ such that
\begin{equation}\label{eq:coercivity}
\int_M\langle \nabla_g u,\nabla_g v\rangle_g + b\, u v \; dV_g \geq \mu \int_M \langle \nabla_g u,\nabla_g v\rangle_g + uv \; dV_g
\end{equation}
for every $u,v\in H_g^1(M)$.

Let $\mathcal{F}$ be a singular Riemannian foliation on $M$ with closed leaves and nontrivial, meaning that $\fol$ is different to the foliation that consists of only one leaf, $\{M\}$, and to the foliation $\{\{p\}\mid p\in M\}$ (see Section~\ref{Sub: Singular Riemannian foliations} for definitions). We will further assume that all the leaves of $\fol$ have dimension greater than or equal to one. For that, we define the number
\begin{equation}\label{eq:kappa_F}
\kappa_\fol:=\min\{\dim L\mid L\in\mathcal{F}\}.
\end{equation}
We say that a function $u\colon M\to \R$ is $\mathcal{F}$-invariant if $u$ is constant on each leaf $L$ of $\fol$.  

\begin{theorem}\th\label{Theorem Main}
Let $(M,g,\fol)$ be an $m$-dimensional Riemannian manifold, $m\geq 3$, together with a nontrivial closed singular Riemannian foliation such that $\kappa_\fol \geq 1$. Assume that $-\Delta_g + b$ is coercive in $H_g^1(M)$, that $b$ and $c$ are $\fol$-invariant functions, with $c>0$ and $2<p\leq2_m^\ast := \frac{2m}{m-2} $. Then \eqref{Main:Yamabe type} admits an infinite number of $\fol$-invariant  solutions, one of them is positive and has least energy among any other such solutions. That is, it attains \eqref{def:LeastEnergy} given below, and the rest are sign-changing solutions. 
\end{theorem}

\begin{remark}
We point out that if \eqref{Main:Yamabe type} admits a constant function as solution, then this is a foliated solution and hence is obtained applying \th\ref{Theorem Main}. In case \eqref{Main:Yamabe type} does not admit a constant function as solution, then the solution given by \th\ref{Theorem Main} is a solution with non trivial geometry.
\end{remark}

\begin{remark}\th\label{R: b constant implies coercive}
Observe that when $b$ is a positive function then the operator $-\Delta_g+b$ is coercive, since we can take 
\[
\mu = \min\{1,\inf\{b(x)\mid x\in M\}\}.
\]
\end{remark}

As an application of \thref{Theorem Main} and \th\ref{R: b constant implies coercive} we give new sign changing solutions to the Yamabe problem and a positive minimal energy solution on any compact manifold of positive scalar curvature that admits a singular foliation:

\begin{corollary}\th\label{Main Corollary Manifolds}
Let $\fol$ be a nontrivial singular Riemannian foliation with closed leaves on a compact Riemannian manifold $(M,g)$.  Assume that $1\leqslant \kappa_\fol$, $c>0$ is constant and that $b>0$ is $\fol$-invariant. Then there exists an infinite number of sign changing $\fol$-invariant solutions to  \eqref{Main:Yamabe type}. In particular, this is true for the Yamabe problem:
\begin{equation}
	-\Delta_g u + S_g u = c \vert u\vert^{\frac{4}{m-2}}u \quad\text{on }\ M,
\end{equation}
provided that the scalar curvature of $(M,g)$, $b=S_g>0$ is $\fol$-invariant. 
\end{corollary}

Observe that for a singular Riemannian foliation with positive dimensional leaves with respect to a Riemannian metric of constant scalar curvature all the conditions in  \th\ref{Main Corollary Manifolds} are satisfied. Due to the fact that the Yamabe problem has been extensively studied for the unit sphere in $\R^n$, and the extensive examples of singular Riemannian foliations for this manifold (see Section~\ref{SS: Examples SRF}) we restate the previous corollary specifically for the case of the unit sphere:

\begin{corollary}\th\label{Main Corollary}
If $c>0$ is constant, then for any nontrivial singular Riemannian foliation $\fol$ on the round sphere $(\Sp^m,g)$ with $1\leq\kappa_\fol$, there exist an infinite number of sign changing $\fol$-invariant solutions to the Yamabe problem: 
\begin{equation}
	-\Delta_g u + \frac{m(m-1)}{4} u = c \vert u\vert^{\frac{4}{m-2}}u \quad\text{on }\ \Sp^m.
\end{equation}
\end{corollary}

\th\ref{Theorem Main} and  Corollaries \ref{Main Corollary Manifolds} and \ref{Main Corollary} directly expand the results obtained in \cite{ClappFernandez2017,FernandezPetean(2020),FernandezPalmasPetean(2020)}. In \cite{ClappFernandez2017} solutions to \eqref{Main:Yamabe type} were found for foliations arising from closed subgroups of isometries acting on a given Riemannian manifold, meanwhile  \cite{FernandezPetean(2020),FernandezPalmasPetean(2020)} provided solutions for foliations arising from isoparametric functions (codimension one foliations). We point out that by the work of Radeschi \cite{Radeschi2014} and of Farrell and Wu \cite{FarrellWu2018}, the notion of a singular Riemannian foliation is more general than the one of a group action and a Riemannian submersion; that is, there are singular Riemannian foliations of arbitrary dimension which cannot arise from a group action nor from a Riemannian submersion (see Subsection~\ref{SS: Examples SRF} for a description of Radeschi's examples). Hence \th\ref{Theorem Main} gives a strict generalization of the current literature in finding sign-changing solutions to problem \eqref{Main:Yamabe type}.

We stress out that \th\ref{Theorem Main} gives an $\fol$-invariant solution with minimal energy  among any other $\fol$-invariant solution. In the absence of symmetries, even if the Yamabe invariant is always attained, problem \eqref{Main:Yamabe type} may not have a ground state. That is, there may not be a solution with minimal energy, as it was shown for instance in  \cite[Theorem~1.5]{ClappFernandez2017}. For the case when $b$ equals the scalar curvature of $M$ it is not clear if the energy minimizing solutions  from \th\ref{Main Corollary Manifolds} attain the Yamabe constant.

Recently in \cite{ClappPistoia2021,ClappSaldanaSzulkin2019}, given a compact Lie group $G$  acting by isometries with positive dimensional orbits, the authors give $G$-invariant solutions to the Yamabe problem which have minimal energy among any other $G$-invariant sign-changing solutions with a fixed number of nodal domains,  by showing the existence of regular optimal $G$-invariant partitions for an arbitrary number of components. This is a more controlled way of finding sign-changing solutions to the Yamabe equation, and highlights the noncompactness of the set of sign-changing solutions. We point out that for singular Riemannian foliations given by group actions it is not clear if there exists a relation between the solutions in  \cite{ClappPistoia2021,ClappSaldanaSzulkin2019} and the ones given by \th\ref{Theorem Main}.

To prove \th\ref{Theorem Main} we state and prove a Rellich–Kondrachov embedding theorem for the subspace of foliated Sobolev functions in \th\ref{Theorem:Sobolev Embedding} below, and a Principle of Symmetric criticality for the energy functional associated to problem \eqref{Main:Yamabe type} in \th\ref{Theorem:Symmetric Criticality}. This is a generalization of Palais' Principle of Symmetric criticality \cite{Palais1979}, and also generalizes the work of Henry \cite{Henry2019} in codimension one foliations. We point out that \th\ref{Theorem:Sobolev Embedding} has been proven independently by Alexandrino and Cavenaghi \cite{AlexandrinoCavenaghi2021}. Nonetheless there are examples of foliations for which \th\ref{Theorem:Symmetric Criticality} still holds (see Section~\ref{SS: Examples SRF}) but not the one in \cite{AlexandrinoCavenaghi2021}. We point out that a general statement of the Palais' Principle of Symmetric criticality for an arbitrary foliated functional is probably false in general. Nonetheless \th\ref{Theorem:Symmetric Criticality} can be applied to the classical Yamabe equation in the context of the foliated Kazdan-Warner problem, that is, finding Riemannian metrics with scalar curvature equal to a prescribed $\fol$-invariant function for a singular Riemannian foliation $(M,\fol)$ as in \cite[Section~2]{CavenaghiDoOSperanca2021}. This problem has also been considered for regular Riemannian foliations on closed manifolds in \cite{WangZhang2013}.

When we have two singular Riemannian foliations $\fol_1$ and $\fol_2$ on a fixed Riemannian manifold $(M,g)$, it is difficult to know if there is a relation between the $\fol_1$-invariant and the $\fol_2$-invariant sign-changing solutions to \eqref{Main:Yamabe type} given by our method. We say that $\fol_1\subset \fol_2$ if for any leaf $L_1\in \fol_1$ there exists a leaf $L_2\in \fol_2$ such that $L_1\subset L_2$. Thus we have the following question:

\begin{question}
Consider a fixed Riemannian manifold $(M,g)$ and two singular Riemannian foliations $\fol_1$ and $\fol_2$ with respect to $g$ such that $\fol_1\subset \fol_2$. Does there exist a solution to  \eqref{Main:Yamabe type} which is $\fol_1$-invariant but  not $\fol_2$-invariant?
\end{question}

As a particular example, note that for a fixed Riemannian manifold $(M,g)$ and a  given  group action by a compact Lie group $G$ via isometries on $(M,g)$, and  a fixed closed subgroup $H\subset G$, then the $H$-orbits are contained in the $G$-orbits. That is we have $\fol_H = \{H(p)\mid p\in M\}\subset \fol_G = \{G(p)\mid p\in M\}$. Even in this homogeneous setting it is not clear how to answer the previous question. Although due to the work of Galaz-Garcia, Kell, Mondino, Sosa \cite[Theorem~5.12]{Galaz-GarciaKellMondinoSosa2018} the subspace of $G$-invariant Sobolev functions $H^1_g(M)^G\subset H^1_g(M)$ is isometric to the Sobolev space $H^1(M/G,\mu^\ast_G)$, where $\mu^\ast_G$ is the pushfoward measure of the Riemannian volume  $V_g$ via quotient the map $M\to M/G$. Thus from our proof it is clear that if $H^1(M/H,\mu^\ast_H)$ is isometric to $H^1(M/G,\mu^\ast_G)$ then we have a negative answer to the question posted above. That is, the $H$-invariant sign-changing solutions given by \th\ref{Theorem Main} are the same as the $G$-invariant solutions given by the theorem.

Our work is organized as follows. In Section~\ref{Sec: Variational Setting} we give the proof of \th\ref{Theorem Main} together with all the necessary concepts to write the proof. In Section~\ref{sec:Preliminaries Foliations} we present some preliminaries of singular Riemannian foliations, give a proof of \th\ref{Lemma:Infinite dimensional}, which says that $H^1_g(M)^\fol$ is an infinite dimensional subspace of $H_g^1(M)$, and finalize presenting examples of singular Riemannian foliations which are not coming from group actions or induced by isoparametric functions (i.e. are codimension one foliations). In Section~\ref{Sec:Symmetryc Criticality} we prove \th\ref{Theorem:Symmetric Criticality}, which is a kind of Principle of Symmetric Criticality for the energy functional associated to our problem, $J$. In Section~\ref{Sec:Compactness} we give a proof of \th\ref{Theorem:Sobolev Embedding} which is a foliated version of the Sobolev and Rellich–Kondrachov embedding theorems. We also prove \th\ref{Theorem:Compactness} which states that $J$ satisfies some Palais-Smale condition. We end with Section~\ref{Section:Variational Principle}, where, employing a variational principle for sign-changing solutions, we show the existence of arbitrarily large number of  critical points for the energy functional in $H^1_g(M)^\fol$ (\th\ref{Theorem:Variational Principle}), which due to \th\ref{Theorem:Symmetric Criticality} are the sign-changing solutions to \eqref{Main:Yamabe type}.

\section*{Acknowledgments}

We thank Monica Clapp for useful conversations about the question posted in the introduction. We also thank Jimmy Petean for useful comments.

\section{Proof of Theorem \ref{Theorem Main}}\label{Sec: Variational Setting}

In this section we briefly state some definitions and results needed to prove our main theorem and we conclude the section by giving its proof.

\subsection{Some words about singular Riemannian foliations}\label{Sub: Singular Riemannian foliations}

Here we define some basic things of singular Riemannian foliations, and provide more details in Section \ref{sec:Preliminaries Foliations}.

\bigskip 
In what follows, $(M,g)$ will be a closed (compact and without boundary) Riemannian manifold of dimension $m\geq 3$. By a \emph{singular Riemannian foliation} $\mathcal{F}$ of $(M,g)$ we mean a decomposition of $M$ into connected injectively immersed submanifolds, called \emph{leaves}, which may have different dimensions, satisfying:
\begin{enumerate}[label = (\roman*)]
	\item $\mathcal{F}$ is a transnormal system, i.e., every geodesic orthogonal to one leaf remains orthogonal to all the leaves that it intersects.
	\item $\mathcal{F}$ is smooth, that is, for every leaf $L\in\mathcal{F}$ and $p \in L$, 
	\[
	T_pL= \mathrm{span}\{ X_p\mid X\in\mathcal{X}_\mathcal{F} \},
	\]
	where $\mathcal{X}_\mathcal{F}$ is the module of smooth vector fields on $M$ which are everywhere tangent to the leaves of $\mathcal{F}$. 
\end{enumerate}

\begin{remark}
Observe that for a compact leaf, an injective immersion is the same as an embedding. Thus for foliations whose leaves are all compact, we consider the leaves embedded in $M$.
\end{remark}

\begin{remark}
It is conjectured that (i) implies (ii), that is, a transnormal system is a smooth foliation. For some discussion on this conjecture see \cite[Final Remarks (c)]{Wilking2007}.  
\end{remark}

\begin{remark}
Condition (i) is independent of (ii), since there are singular foliations satisfying (ii) but not (i). For example, in \cite{Sullivan1976} a $5$-dimensional manifold with a smooth foliation by $1$-dimensional circles, that does not have finite holonomy is given. By \cite[Theorem~2.6]{Moerdijk} any foliation with compact leaves,  all of the same dimension, which admits a Riemannian metric satisfying (i) has finite holonomy. Thus, the smooth foliation given in  \cite{Sullivan1976} does not admit a Riemannian metric satisfying (i). 
\end{remark}

The leaves of maximal dimension are called \emph{regular} and the other ones \emph{singular}. If all the leaves are regular, then $\mathcal{F}$ is called a \emph{regular Riemannian foliation}. The \emph{dimension} of $\mathcal{F}$, denoted by $\dim\mathcal{F}$, is the dimension of the regular leaves and the \emph{codimension of $\fol$} is equal to $\dim M - \dim\fol$.

\subsection{Variational Setting}
Let $(M,g)$ be a closed Riemannian manifold of dimension $m\geq 3$, $p\in(2,2^\ast)$, and let $\fol$ be a singular Riemannian foliation on $M$. In what follows, $\langle \cdot,\cdot\rangle_g$ will denote the Riemannian metric in $M$, while $\nabla_g u$ will denote the gradient  with respect to $g$ of a smooth function $u\colon M\to\R$. The Sobolev space $H_g^1(M)$ is the closure of $\mathcal{C}^\infty(M)$ under the norm, $\Vert \cdot \Vert_{H^1_g(M)}$, induced by the interior product
\begin{equation}\label{Eq:Sobolev Norm}
\langle u,v \rangle_{H^1_g(M)}:=\int_{M} \langle \nabla_g u,\nabla_g v\rangle_g + uv\;dV_g,\quad u,v\in\mathcal{C}^\infty(M).
\end{equation}
By density, this bilinear form extends to a continuous symmetric bilinear form on $H_g^1(M)\times H_g^1(M).$ Even if, formally, $\langle\nabla_g u,\nabla_g v\rangle_g$ is not defined for an arbitrary weak gradient, we will understand \eqref{Eq:Sobolev Norm} as a limit and conserve the notation on the right hand side for every $u,v\in H_g^1(M)$.

Fix two functions $b, c \in C^\infty(M)$ with $c>0$ and suppose that the operator $-\Delta+b$ is coercive
(see \eqref{eq:coercivity}). Since $-\Delta+b$ is coercive, the bilinear product
\begin{equation}\label{eq-innerprod-b}
\langle u, v\rangle_b:= \int_M \langle \nabla_g u,\nabla_g v \rangle_g + b\, uv\ dV_g,\qquad u,v\in H_g^1(M)
\end{equation}
is definite positive. Thus, it is an inner product and consequently induces a norm in $H_g^1(M)$, which furthermore is 
equivalent to the norm $\Vert \cdot \Vert_{H^1_g(M)}$ given that
$M$ is compact, $b$ is continuous and $-\Delta+b$ is coercive.

Since $c$ is positive and continuous and $M$ is compact, we also have a norm in the $L^p_g(M)$ space,
\begin{equation}\label{eq-norm-cp}
\vert u\vert_{c,p}:= \left(\int_M c\vert u\vert^p \; dV_g\right)^{1/p} \quad u \in L^p_g(M), 
\end{equation}
equivalent to the standard norm of $L^p_g(M)$. Indeed, denoting by $K$ the maximum of $c$, and by $k$ its minimum, we have that $K\geqslant k >0$, and thus
\[
    k^{1/p}\Vert u\Vert_{L^p_g(M)} \leqslant \vert u\vert_{c,p} \leqslant K^{1/p}\Vert u\Vert_{L^p_g(M)}.
\]

The energy functional $J \colon H^1_g(M)\to\R$ given by
\[
J(u):= \frac{1}{2}\Vert u\Vert_b^2 - \frac{1}{p}\vert u\vert_{c,p}^p
\]
is a well defined $C^2$ functional for all $p\in(2,2^\ast_m]$, and its derivative is given by
\begin{equation}\label{Eq:Derivative J}
J'(u)v = \langle u,v \rangle_b - \int_M c \vert u\vert^{p-2}uv\, dV_g,\qquad u,v\in H^1_g(M). 
\end{equation}
Moreover, $u \in H^1_g(M)$ is a
critical point of $J$ if and only if $u$ is a solution to \eqref{Main:Yamabe type}. The nontrivial critical points of $J$ lie in the Nehari manifold
\begin{linenomath}
\begin{equation}\label{eq-defNehari}
\begin{split}
   \mathcal{N}_g &:= \{u\in H_g^1(M)\mid u\neq 0 , J'(u)u=0\}\\
   &= \{u\in H_g^1(M)\mid u\neq 0 , \Vert u\Vert_b^2 = \vert u\vert_{c,p}^p\}, 
\end{split}
\end{equation}
\end{linenomath}
which is a closed $C^1$ Hilbert submanifold of $H^1_g(M)$ of codimension one and radially diffeomorphic to $\Sp_{H_g^1(M)}$, the unit sphere in $H^1_g(M)$ with respect to the norm $\Vert\cdot\Vert_b$. Concretely, for each $u\in H_{g}^1(M)\smallsetminus\{0\}$, there exists a unique $t_u\in(0,\infty)$ such that $t_u u\in \mathcal{N}_g$, and it is given explicitly by
\[
t_u := \left(\frac{\Vert u\Vert_b^2}{\vert u\vert_{c,p}^p}\right)^{1/(p-2)}. 
\]
 Hence, we can define a projection $\sigma\colon H_g^1(M)\smallsetminus\{0\}\to \mathcal{N}_g$ given by
\begin{equation}\label{eq: Projection Nehari}
\sigma(u) := t_u u = \left(\frac{\Vert u\Vert_b^2}{\vert u\vert_{c,p}^p}\right)^{1/(p-2)} u,
\end{equation}
which induces the radial diffeomorphism $\sigma\colon\Sp_{H_g^1(M)}\to \mathcal{N}_g$ 
with inverse $u\mapsto \frac{u}{\Vert u\Vert_b}$.  In addition, if $u\in\mathcal{N}_g$, then 
\begin{equation}\label{Eq:Nehari maximum}
    J(u)=\max\{J(tu)\mid t\geqslant 0\},
\end{equation}
(see \cite[Lemma 4.1]{Willem1996}).

For any singular Riemannian foliation, define the space
\[
\mathcal{C}^\infty(M)^\fol:=\{u\in\mathcal{C}^\infty(M)\mid u \text{ is constant on } L \text{ for any }L\in\fol\}
\]
and let $H_g^1(M)^\fol$ be the closure of $\mathcal{C}^\infty(M)^\fol$ under the norm $\Vert\cdot\Vert_{H^1_g(M)}$. 
Hence $H_g^1(M)^\fol$ is a closed subspace of $H_g^1(M)$ and $u\in H_g^1(M)^\fol$ if and only if $u\in H_g^1(M)$ and is  $dV_g$-a.e. constant on the leaves of $\fol$.

For $1\leqslant \kappa_\fol<m$ with $\kappa_\fol$ defined as in \eqref{eq:kappa_F} we show in \th\ref{Lemma:Infinite dimensional} that $H^1_g(M)^\fol$ is an infinite dimensional Hilbert space endowed with the restriction of the interior product $\langle\cdot,\cdot \rangle_{H^1_g(M)}$. The critical points of $J$ that are contained in $H_g^1(M)^\fol$ will be called $\fol$-invariant critical points of $J$. Hence, the nontrivial $\fol$-invariant critical points of $J$ lie in the restricted Nehari manifold
\[
\mathcal{N}_g^\fol:=\mathcal{N}_g\cap H_g^1(M)^\fol, 
\]
which is nonempty by virtue of Theorem \ref{Lemma:Infinite dimensional} below and the definition of the projection \eqref{eq: Projection Nehari}. 
The least energy of the functions contained in the restricted Nehari manifold is given by 
\begin{equation}\label{def:LeastEnergy}
\tau_g^\fol := \inf \{ J(u) \mid u \in \mathcal{N}_g^\fol\},    
\end{equation}
which is a positive number (see \cite[Theorem 4.2]{Willem1996}).
We recall that the positive solution obtained in \th\ref{Theorem Main} will attain this number, and thus it will be a solution of least energy.

\subsection{Main results to prove Theorem~\ref{Theorem Main} and proof of Theorem~\ref{Theorem Main}}\label{Sec:Main Results}
The main ingredients for the proof of \th\ref{Theorem Main} are 
\th\ref{Lemma:Infinite dimensional}, a compactness result (\th\ref{Theorem:Compactness}), a Principle of Symmetric Criticality for the functional $J$ restricted to the space $H_g^1(M)^\fol$ (\th\ref{Theorem:Symmetric Criticality}), and a variational principle (\th\ref{Theorem:Variational Principle}), which we state below. 

\bigskip
First we notice that $H_g^1(M)^\fol$ is an infinite dimensional space, see Section~\ref{sec:Preliminaries Foliations} for the proof. Recall that $\kappa_\fol:=\min\{\dim L\mid L\in\mathcal{F}\}$.

\begin{theorem}\th\label{Lemma:Infinite dimensional}
Let $(M,g,\fol)$ be an $m$-dimensional closed Riemannian manifold together with a nontrivial singular Riemannian foliation with $1\leq \kappa_\fol$. Then, for any $k\in\mathbb{N}$, there exist nontrivial nonnegative functions $u_1,\ldots,u_k\in\mathcal{C}^\infty(M)^\fol$ with pairwise disjoint supports. In particular, it follows that $H_g^1(M)^\fol$ is infinite dimensional.
\end{theorem}

Given $\tau \in \mathbb R$, we say that a sequence $(u_n)_{n\in\mathbb{N}}$ in $H_g^1(M)^\fol$ is a \emph{$(PS)^\fol_\tau$-sequence for $J$} if $J(u_n)\to \tau$ and $J'(u_n)\to 0$ in $(H_g^1(M)^\fol)^{\ast}$, where the last space is the \emph{dual of $H_g^1(M)^\fol$}. We say that \emph{$J$ satisfies the $(PS)^\fol_\tau$-condition in $H_g^1(M)^\fol$} if every 
$(PS)^\fol_\tau$-sequence for $J$ has a strongly convergent subsequence in $H^1_g(M)$.  With these definitions, the statement of the compactness results is as follows.

\begin{theorem}[Compactness]\th\label{Theorem:Compactness}
Under the hypotheses of \th\ref{Theorem Main}, except that we also allow $b$ and $c$ to be non foliated, the functional $J$ satisfies the $(PS)^\fol_\tau$-condition for every  $\tau\in\R$.
\end{theorem}

We prove this result in Section \ref{Sec:Compactness}. This is a
generalization of a classic result by E. Hebey and M. Vaugon \cite{HebeyVaugon1997} and a consequence of a Sobolev embedding theorem for singular Riemannian foliations (\th\ref{Theorem:Sobolev Embedding}). We believe this result is interesting in its own right, since previous versions of it have been applied to restore compactness in many variational problems involving symmetries  given by compact groups actions via isometries (see for instance, \cite{ClappFernandez2017,BartschSchneiderWeth2004,ClappTiwari2016}). 

With this at hand, in Section \ref{Section:Variational Principle} we will be able to generalize Clapp-Pacella's variational principle in \cite{ClappPacella2008} (see also \cite{ClappFernandez2017,ClappWeth2005}). 

\begin{theorem}[Variational Principle]\th\label{Theorem:Variational Principle}
Under the hypotheses of Theorem~\ref{Theorem Main}, let $W$ be a nontrivial finite dimensional subspace of $H_g^1(M)^\fol$. If $J$ satisfies the $(PS)^\fol_\tau$ condition in $H_g^1(M)^\fol$ for every $\tau\leq\sup_W J$, then $J$ has at least one positive critical point $u_1$ and $k-1:=\dim W -1$ pairs of sign-changing critical points $\pm u_2,\ldots,\pm u_k$ in $H_g^1(M)^\fol$ such that $J(u_1)=\tau_g^\fol$ and $J(u_i)\leq \sup_W J$ for $i=2,\ldots,k$.
\end{theorem} 

The Principle of Symmetric Criticality needed to prove our main theorem is the following.  

\begin{theorem}[Principle of Symmetric Criticality for $J$]\th\label{Theorem:Symmetric Criticality}
Under the hypotheses of Theorem~\ref{Theorem Main}, if $u\in H_g^1(M)^\fol$ is a critical point of the functional $J$ restricted to $H_g^1(M)^\fol$, then $u$ is a critical point of $J$ in the space $H_g^1(M)$.
\end{theorem}

The proof of this result appears in Section \ref{Sec:Symmetryc Criticality}. In \cite{Henry2019}, a similar argument is used for the case of foliations given by isoparametric functions.

\bigskip
We now prove \th\ref{Theorem Main} in an analogous fashion to the first Main Result in \cite{ClappFernandez2017}.

\begin{proof}[Proof of Theorem \ref{Theorem Main}] 
In order to find nontrivial foliated critical points of $J$, by Theorem \ref{Theorem:Symmetric Criticality}, it suffices to consider critical points of the restriction of $J$ to the space $H_g^1(M)^\fol$. To do so, we proceed as follows. By Theorem \ref{Lemma:Infinite dimensional},
since we are assuming that the singular Riemannian foliation is nontrivial,  for any
given $k\in\mathbb{N}$, we may choose $k$ nontrivial functions $\omega_{i}%
\in\mathcal{C}^\infty(M)^\fol$ $i=1,\ldots,k$ with disjoint supports.
Let $W$ be the linear subspace of
$H_g^1(M)^\fol$ spanned by $\{\omega_{1},\ldots,\omega_{k}\}$. Since
$\omega_{i}$ and $\omega_{j}$ have disjoint supports for $i\neq j,$ the set
$\{\omega_{1},\dots,\omega_{k}\}$ is orthogonal in $H_g^1(M)^\fol$.
Hence, $\dim W=k.$

Now, as $\kappa_\fol\geq 1$, Theorem
\ref{Theorem:Compactness} yields that $J$ satisfies $(PS)^\fol_{\tau}$ in $H_g^1(M)^\fol$ for every $\tau\in\mathbb{R}$. Therefore, we can apply Theorem
\ref{Theorem:Variational Principle} and get at least one positive and $(k-1)$ sign-changing
$\fol$-invariant critical points for the restriction of $J$  to $H_g^1(M)^\fol$. Theorem \ref{Theorem:Symmetric Criticality} yields that these are also critical points of $J$ in $H_g^1(M)$ and, hence nontrivial $\fol$-invariant solutions to problem
\eqref{Main:Yamabe type}. Given that  $k\in\mathbb{N}$ is arbitrary,
we conclude that there are infinitely many sign-changing solutions.
\end{proof}

\section{Singular Riemannian Foliations}\label{sec:Preliminaries Foliations}

In the first part of this section we provide the material needed in Section \ref{Sec:Compactness} to show a Sobolev embedding theorem (Theorem \ref{Theorem:Sobolev Embedding}), such as the  existence of smooth and $\fol$-invariant partitions of unity and the Slice Theorem. Furthermore,  we prove that $H^1_g(M)^\fol$ is infinite dimensional (Theorem \ref{Lemma:Infinite dimensional}). In the second part we give examples of singular Riemannian foliations. In particular, we present examples for which previous methods to find solutions to \eqref{Main:Yamabe type} such as those \cite{ClappFernandez2017,FernandezPetean(2020),FernandezPalmasPetean(2020),AlexandrinoCavenaghi2021} do not apply.

\subsection{Singular Riemannian Foliations Revisited}

\bigskip 

Let $\mathcal{F}$ be a singular Riemannian foliation with closed leaves on a closed Riemannian manifold $(M^m,g)$. Given $q\in M$, denote by $L_q$ the leaf containing $q$, by $k_q$ the dimension of $L_q$ and by $\pi_q\colon \R^{k_q}\times\R^{m-k_q}\to\R^{m-k_q}$ the natural projection. The following lemma describes the foliation in a small neighborhood of $q$.

\begin{lemma}[Proposition~2.17 in \cite{Radeschi-notes}, Figure~\ref{fig:Plaque Trivialization}]\th\label{Lemma:Charts of foliation around a point}
For any $q\in M$, there exists a coordinate system $(W,\varphi)$  such that
\begin{enumerate}
\item $q\in W$, $\varphi(W)=U\times V$, where $U\subset\R^{k_q}$ and $V\subset\R^{m-k_q}$ are open and bounded subsets with smooth boundary;
\item for any $q'\in W$, $U\times\pi_q(\varphi(q'))\subset\varphi(L_{q'}\cap W)$.
\end{enumerate}
\end{lemma}

\begin{figure}[h!]
    \centering
    \includegraphics[scale =0.3]{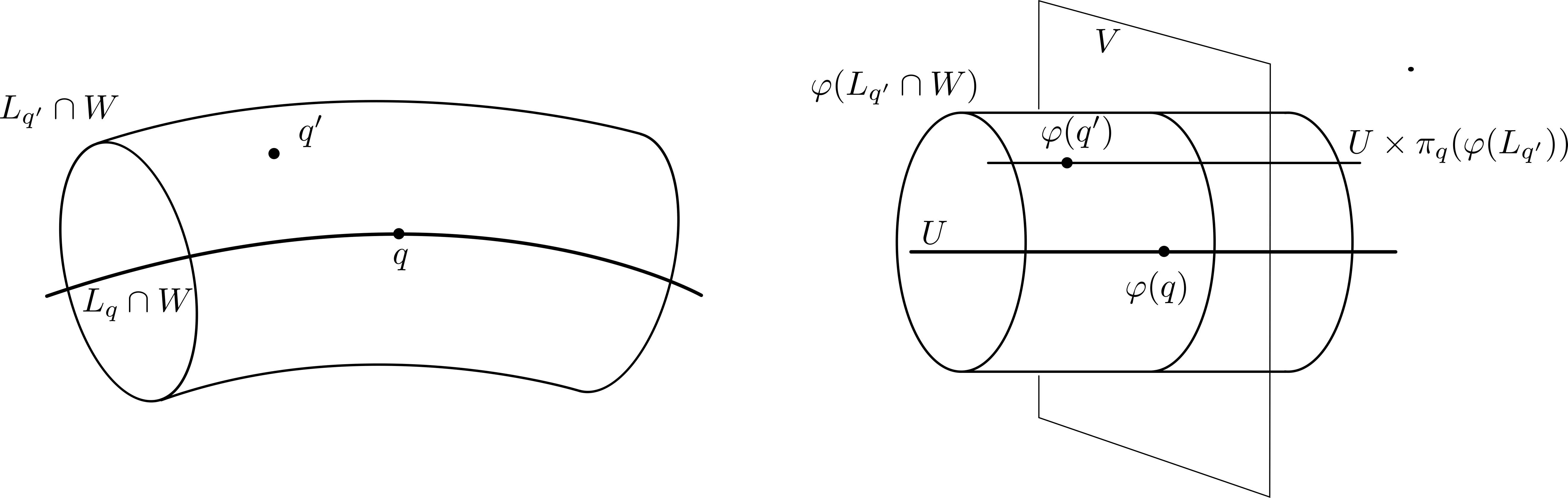}
    \caption{Local trivialization of a foliation}
    \label{fig:Plaque Trivialization}
\end{figure}

Consider $p\in M$ fixed. Take $\varepsilon>0$ such that $\expo_p\colon B_\varepsilon(0)\subset T_p M\to M$ is a diffeomorphism. Denote by $\D^\perp_p(\varepsilon)\subset \nu_p(M,L_p)$  the closed ball of radius $\varepsilon$ in the normal space to the leaf at $p$. Consider $S_p(\varepsilon) = \expo_p(\D^\perp_p(\varepsilon))$. 
For $q \in S_p(\varepsilon)$ denote by $\tilde{L}_q$ the connected component of $L_q\cap S_p(\varepsilon)$ containing $q$. 
Since $\varepsilon$ is smaller than the injectivity radius at $p$, then $q\in S_p(\varepsilon)$ if and only if $q= \expo_p(v)$ for a unique $v\in \D^\perp_p(\varepsilon)$. For such $q$ we
set $\mathcal{L}_v = \expo_p^{-1}(\tilde{L}_q)$. The following theorem states that the partition of $\D^\perp_p(\varepsilon)$ by $\fol_p^\perp(\varepsilon) = \{\mathcal{L}_v\mid v\in \D^\perp_p(\varepsilon) \}$ is a singular Riemannian foliation.

\begin{thm}[Infinitesimal Foliation, Proposition~6.5 in \cite{Molino}]\th\label{T: Infintisimal foliation}
    Let $(M,\fol)$ be a singular Riemannian foliation on the compact manifold $M$ and fix $ p\in M$. Take $\varepsilon>0$ smaller than the injectivity radius of $M$ at $p$. Then   $(\D_p^\perp(\varepsilon), \fol_p^\perp(\varepsilon))$ equipped with the Euclidean metric is a singular Riemannian foliation.
\end{thm}

Moreover, $(\D_p^\perp(\varepsilon), \fol_p^\perp(\varepsilon))$ does not depend on the radius chosen. 

\begin{lemma}[\cite{Radeschi-notes}]\th\label{L: Homothetic transformation Lemma}
 Let $\varepsilon>0$ and $\lambda>0$ be such that   $\varepsilon$ and $\lambda \varepsilon$ are smaller than the injectivity radius of $M$ at $p$. Then the map $h_\lambda\colon \D^\perp_p(\varepsilon)\to \D^\perp_p(\lambda\varepsilon)$ given by  $h_\lambda(v) = \lambda v $ is a foliated diffeomorphism between $(\D_p^\perp(\varepsilon), \fol_p^\perp(\varepsilon))$ and $(\D_p^\perp(\lambda\varepsilon), \fol_p^\perp(\lambda\varepsilon))$
\end{lemma}

The previous result implies that for $\varepsilon>0$ small enough  $\fol_p^\perp(\varepsilon)$ is independent of $\varepsilon$, and thus we set
$\fol_p^\perp = \fol_p^\perp(\varepsilon)$.
We call $\fol_p^\perp$ the \emph{infinitesimal foliation at $p$}.

Observe that the image under $\expo_p$ of different leaves of $\fol_p^\perp$ may be contained in the same leaf of $\fol$. To account for this, we define the holonomy group of a leaf.

\begin{lemma}[See \cite{Corro,Radeschi15,Radeschi-notes}]\th\label{L: action of curves} 
Under the hypotheses of Theorem \ref{T: Infintisimal foliation} and denoting by $L$ the leaf of $\fol$ that contains $p$, the following statements hold:
\begin{enumerate}[label=\alph*)]
\item Given a piece-wise smooth curve $\gamma\colon [0,1]\to L$ starting  at $p$  there exists a map $G(\gamma)(t)\colon \D^\perp_p(\varepsilon)\to \D^\perp_{\gamma(t)}(\varepsilon)$ which is a foliated isometry between $(\D^\perp_p(\varepsilon),\fol_p^\perp)$ and $(\D^\perp_{\gamma(t)},\fol^\perp_{\gamma(t)})$.
\item The set 
\[
\Hol(L,p) = \{G(\gamma)(1)\mid \gamma \mbox{ is a closed loop in } L \mbox{ centered at }p\}
\] is a group under the composition operation. 
\item Consider $p\neq q\in L$, and $\alpha\colon [0,1] \to L$ a piece-wise curve starting at $p$ and ending at $q$. Then 
\[
\Hol(L,p) = G(\alpha)(1)^{-1}\Hol(L,q)G(\alpha)(1).
\]
\item Given two piece-wise smooth curves $\gamma_1$ and $\gamma_2$ starting at $p$ and homotopic relative to its end points,  the composition  of their corresponding foliated isometries, $G(\gamma_2)(1)^{-1}\circ G(\gamma_1)(1)$, is homotopic to the identity map.
\end{enumerate}
\end{lemma}

The group $\Hol(L,p)$ given in \th\ref{L: action of curves} is called the \emph{holonomy group of the leaf $L$ at $p$}. For any other $q\in L$  due to  $c)$ and $d)$ the groups $\Hol(L,p)$ and $\Hol(L,q)$ are conjugates, and we denote their conjugacy class as $\Hol(L)$. The regular leaves $L$ of $\fol$ for whom $\Hol(L) =[\{Id\}
]$ are called \emph{principal leaves}. 
We denote by $M_{\prin}\subset M$ the set of points contained in the principal leaves. We remark that this is an open and dense subset of $M$.

In order to prove that $H_g^1(M)^\fol$ is infinite dimensional, we define 
the \emph{leaf space} of $(M,g,\fol)$ to be the quotient space induced by the partition $\fol$ equipped with the quotient topology and denote it by $M^\ast=M/\fol$. We denote the quotient map as $\pi\colon M\to M/\fol$. Given a subset $A\subset M$, we denote its image $\pi(A)$ by $A^\ast$.

\begin{proof}[Proof of  \th\ref{Lemma:Infinite dimensional}]
Let $(M,\fol)$ be  a singular Riemannian foliation with closed leaves on the complete manifold $M$. Then the leaf space $M/\fol$ is a complete length metric space. Thus for any open cover $\{U^\ast\}$ of $M/\fol$ we have a subordinate partition of unit $\phi_i^\ast\colon M/\fol\to \R$ such that $\supp(\phi_i^\ast)\subset U_i^\ast$. We can refine the cover $\{U_i^\ast\}$ so that for any $k\in \N$, at least $k$ open neighborhoods of the cover are disjoint. Taking $\pi\colon M\to M/\fol$ and setting $u_j = \phi_j^\ast\circ\pi$ we get the $k$ desired continuous foliated functions with disjoint support.

This approach can be improved to get smooth foliated functions.  We can restrict $\pi\colon M\to M/\fol$ to $M_{\prin} \subset M$ and $M_{\prin}^\ast= M_{\prin} /\fol \subset M/\fol$, so that $\pi\colon M_{\prin}\to M_{\prin}^\ast$. We endow $M_{\prin}$ with $g|_{M_{\prin}}$ the induced metric by the inclusion into $M$. Then $(M_{\prin},g|_{M_{\prin}},\fol|_{M_{\prin}})$ is a regular Riemannian foliation. Moreover the leaf space $M^\ast_{\prin}$ is a manifold. By Theorem \cite{Gromoll} there is a Riemannian metric $g^\ast_{\prin}$ on $M^\ast_{\prin}$  such that the map $\pi\colon (M_{\prin},g|_{M_{\prin}})\to (M^\ast_{\prin},g^\ast_{\prin})$ is actually a Riemannian submersion.

Then by restricting an open cover $\{U^\ast_i\}$ of $M/\fol$ to $M_{\prin}^\ast$, we get a smooth partition of unity $\{\phi_i^\ast\colon M_{\prin}^\ast\to \R\}$ subordinated to the cover $\{U_i^\ast\cap M_{\prin}^\ast\}$. Setting $\phi_i = \phi_i^\ast\circ \pi$ gives us a smooth  partition of unity subordinated to the open cover $\{\pi^{-1}(U_i^\ast\cap M_{\prin}^\ast)\}$ of $M_{\prin}$; 
observe that by construction these functions are constant along the leaves. We can refine this open cover so that for any given $k\in \N$, there are at least $k$ functions $\phi_1,\ldots,\phi_k$ which have disjoint support, contained in $M_{\prin}$.  Setting $U_i = \pi^{-1}(U_i^\ast \cap M_{\prin}^\ast)$, we extend these functions trivially on $M\smallsetminus(U_1\cup\ldots \cup U_k)$ to get the desired smooth functions $u_i$, constant along the leaves with disjoint support.
\end{proof}

The following theorem states how a foliated tubular neighborhood looks like:

\begin{thm}[Slice Theorem in \cite{Radeschi15}]\th\label{T: Slice Theorem}
    Let $(M,\fol)$ be a  singular Riemannian foliation with closed leaves. 
    Given a leaf $L\in\fol$ and $p\in L$, denote by $\mathrm{Tub}^\varepsilon(L)$ the tubular neighborhood of radius $\varepsilon$ of $L$. Then  there exist $\varepsilon>0$ small enough and a $\Hol(L,p)$-principal bundle $P \to L_p$, such that $(\mathrm{Tub}^\varepsilon(L), \fol \cap \mathrm{Tub}^\varepsilon(L))$ is foliated diffeomorphic to 
    \[
        (P\times_{\Hol(L,p)} \D^\perp_p(\varepsilon), P\times_{\Hol(L,p)}\fol_p^\perp).
    \]
    Here $P\times_{\Hol(L,p)} \D^\perp(\varepsilon)$ denotes the quotient $(P\times \D^\perp(\varepsilon))/\Hol(L,p)$ with respect to the product action of $\Hol(L,p)$ on $P\times \D^\perp(\varepsilon)$, i.e. $h(p,v)=(ph,h^{-1}v)$. 
\end{thm}

Given $[x,v]\in P\times_{\Hol(L)} \D^\perp_p(\varepsilon)$ the leaf $L_{[x,v]}$ of $P\times_{\Hol(L)} \fol_p^\perp$ consists of all the classes $[y,w]\in P\times_{\Hol(L)}\D^\perp_p(\varepsilon)$ such that $(yh,h^{-1}w)\in P\times \mathcal{L}_v$  for some $h\in \Hol(L)$.  

\begin{lemma}\th\label{L: Existence of foliated Partition of unity}
Let $(M,\fol)$ be a closed foliation on a compact manifold. Then for any $\varepsilon>0$ there exists an open cover $\{Z_i\}_{i\in \Lambda}$ of $M$ such that each $Z_i$ is a tubular neighborhood of radius at most $\varepsilon$ and as in \th\ref{T: Slice Theorem} such that there exists a finite subcover $\{Z_i\}_{i = 1}^N$ and a smooth partition of unity $\{\phi_i\}^N_{i = 1}$ adapted to the subcover which is $\fol$-invariant; i.e. for each $p \in M$ if $q\in L_p$ then $\phi_i(q) = \phi_i(p)$. 
\end{lemma}

\begin{proof}
For each $p\in M$ consider $\varepsilon(p)>0$ 
such that \th\ref{T: Slice Theorem} holds for the tubular neighborhood of $L_p$ of radius $\varepsilon(p)$. If $\varepsilon(p)>\varepsilon$ then we replace $\varepsilon(p)$ by $\varepsilon$, and note that  \th\ref{T: Slice Theorem} still holds for this value. The collection $\{\mathrm{Tub}^{\varepsilon(p)}(L_p)\mid p\in M\}$  is an open cover of $M$ and since $M$ is compact there exist $Z_i=\mathrm{Tub}^{\varepsilon(p_i)}(L_{p_{i}})$, $i=1,\ldots,N$, such that $M = \cup_{i=1}^N Z_i$.  

For each $1\leqslant i \leqslant N$, let $\psi_i\colon Z_i \to P_i\times_{\Hol_{i}}\D^\perp_i$ be the foliated diffeomorphism given by \th\ref{T: Slice Theorem}.
We will define a $(P_i\times_{\Hol_i}\fol_i)$-foliated function $\bar{\phi}_i\colon P_i\times_{\Hol_{i}}\D^\perp_i\to \R$ and 
then define functions $f_i\colon M\to \R$
\begin{linenomath}
\begin{equation*}
 f_i(q) =
 \begin{cases}
 \bar{\phi}_i\circ\psi_i(q)  &  q\in Z_i\\
 0 & q\not\in Z_i,
 \end{cases}   
\end{equation*} 
\end{linenomath}
so that $\phi_i\colon M\to \R$ defined as 
\[
\phi_i(q) =\frac{f_i(q)}{\sum_{j=1}^N f_j(q)} \quad q \in M, 
\]
will be the desired partition of unity. Observe that $f_i$ is $\fol$-in\-va\-ri\-ant since it is a composition of a foliated diffeomorphism with an $\fol$-in\-va\-ri\-ant function, and thus $\phi_i$ is an $\fol$-in\-va\-ri\-ant function.

For each $1\leqslant i \leqslant N$, to define $\bar \phi_i$, first consider a smooth nonnegative decreasing function $\tilde{\varphi}_i\colon \R\to \R$ such that $\tilde{\varphi}_i(t) = 1$ for $t\leqslant 0$, $\tilde{\varphi}_i(t) = 0$ for $t\geqslant \varepsilon(p_i)$. Let $\overline{\varphi}_i\colon \D^\perp_{p_i}(\varepsilon(p_i))\to \R$ be given by $\overline{\varphi}_i(v) = \tilde{\varphi}_i(\Vert v\Vert)$. Observe that over the spheres in $\D^\perp_{p_i}(\varepsilon(p_i))$ centered at $0$, 
$\overline{\varphi}_i$ is constant. Since the leaves of $\fol_{p_i}^\perp$ are contained in such spheres we conclude that $\overline{\varphi}_{i}$ is $\fol_{p_i}^\perp$-foliated.

Set $\Hol_i = \Hol(L_{p_i},p_i)$, $\D^\perp_i = \D^\perp_{p_i}(\varepsilon(p_i))$, $\fol_i = \fol_{p_i}^\perp$ and let $P_i$ be the total space of the principal $\Hol_i$-bundle from \th\ref{T: Slice Theorem}. Define 
\[\tilde{\phi}_i\colon P_i\times \D_i^\perp\to \R \quad \textrm{as} \quad \tilde{\phi}_i(x,v)= \bar{\varphi}_i(v).\] 
Now note that for any $h\in \Hol_i$ we have $ \tilde{\phi}_i(h(x,v)) = \tilde{\phi}_i(x,v)$. Indeed, this follows from the fact that $\Vert h^{-1}v \Vert = \Vert v\Vert$ as can be seen below, 
\begin{linenomath}
\begin{align*}
    \tilde{\phi}_i(h(x,v)) = & \tilde{\phi}_i(xh,h^{-1}v) = \bar{\varphi}_i(h^{-1}v)
    = \tilde{\varphi}_i(\Vert h^{-1} v\Vert)\\
    = & \tilde{\varphi}_i(\Vert v\Vert)
    = \bar{\varphi}_i(v)
    = \tilde{\phi}_i(x,v).
\end{align*}
\end{linenomath}
Thus we have a well defined map
\[
\bar{\phi}_i\colon P_i\times_{\Hol_{i}}\D^\perp_i\to \R \quad{as}\quad 
\bar{\phi_i}[p,v]= \bar{\varphi_i}(v),
\]
and we will check using local trivializations that it is also smooth.

 Consider the projection map $\pi_i\colon P_i\to L_{p_i}$ of the principal $\Hol_i$-bundle, and denote by $\bar{\pi}_i\colon P_i\times_{\Hol_{i}}\D^\perp_i\to L_{p_i}$ the projection of the associated bundle. Fix $q\in L_{p_i}$, and consider $U_{\alpha i}\subset L_{p_i}$ a sufficiently small open neighborhood of $q$, such that there exists a trivialization $\tau_{\alpha i}\colon U_{\alpha i}\times \Hol_i\to \pi_i^{-1}(U_{\alpha i}) \subset P_i$ of $\pi_i$. Observe that the map $\bar{\tau}_{\alpha i}\colon U_{\alpha i}\times \D^\perp_i \to \bar{\pi}_i^{-1}(U_{\alpha i}) \subset P_i\times_{\Hol_{i}}\D^\perp_i$ given by 
\[
    \bar{\tau}_{\alpha i}(p,v) = [\tau_{\alpha i}(p,\mathrm{Id}),v],
\]
where  $\mathrm{Id}\in \Hol_i$ denotes the identity element, is a trivialization of the associated bundle $\bar{\pi}_i$. Then we have that $\bar{\phi}$ is given over a local trivialization as
\[
    \bar{\phi}_i\circ \bar{\tau}_{\alpha i}(p,v) = \bar{\varphi}_i(v).
\]
Therefore $\bar{\phi}_i\colon P_i\times_{\Hol_{i}}\D^\perp_i\to \R$ can be written as the composition of smooth maps and so it is also smooth.

We check now that $\bar{\phi}_i$ is $(P_i\times_{\Hol_i}\fol_i)$-foliated. 
Consider $[x,v]\in P_i\times_{\Hol_{i}}\D^\perp_i$, and take $[y,w]\in L_{[x,v]}$. Then there exists some $h\in \Hol_i$ such that $hw\in \mathcal{L}_v$. Then
\begin{linenomath}
\begin{align*}
    \bar{\phi}_i[y,w] &= \tilde{\phi}_i(h(y,w))
    = \tilde{\phi}_i(hy,hw)\\
    &= \bar{\varphi}_i(hw)
    = \bar{\varphi}_i(v)\\
   & = \tilde{\phi}_i(x,v)
    = \bar{\phi}_i[x,v].
\end{align*}
\end{linenomath}
This concludes the construction of $\bar{\phi}_i$ and thus the proof of the lemma.
\end{proof}

\subsection{Examples of singular Riemannian foliations}\label{SS: Examples SRF}

Here we pre\-sent in detail some examples of singular Riemannian foliations already mentioned in the introduction.
\bigskip

\textbf{Riemmannian submersions:} Given a Riemannian submersion $\pi\colon M\to B$ we can define a singular Riemannian foliation where the foliation consists of the set of preimages under $\pi$ of the image of $\pi$. We give more details below. 

Recall that a surjective differentiable map $\pi\colon M\to B$ between smooth manifolds is a \emph{submersion} if at any point $p\in M$, the differential map $D_p\pi\colon T_p M\to T_p B$ is a surjective linear map. This implies that the dimension of $M$ is greater than or equal to the dimension of $B$. From now on assume that the dimension of $M$ is strictly greater than the dimension of $B$. Assume that $M$ has a Riemannian metric $g$, and for any $p\in M$ denote by $V(p)$ the subspace of $T_p M$ tangent at $p$ to the fiber $L_p = \pi^{-1}(\pi(p))$ of $\pi$ through $p$, and let $V$ to be the subbunddle over $M$ with fiber equal to $V(p)$.  We say that $\pi$ is a \emph{Riemannian submersion} if for any $X\in V$ the Lie derivative of $g$ in the direction of $X$ satisfies $\mathcal{L}_X g = 0$, that is, the metric $g$ is invariant in the directions tangent to the fibers. Setting $H(p)$ to be the $g$-orthogonal complement of $V(p)$ in $T_p M$, this implies that for any $U,V\in H(p)$ the inner product given as
\[
    h_{\pi(p)}(\pi_\ast U, \pi_\ast V) = g_p(U,V)
\]
is a Riemannian metric on $B$.

\medskip

\textbf{Homogeneous foliations:} Another large family of examples stem from a compact Lie group $G$ acting by isometries on a given Riemannian manifold $(M,g)$. Then, for the partition consisting of the set of orbits of the action of $G$ on $M$, $\fol$, we have that $(M,g , \fol)$ is a singular Riemannian foliation which is known as a \emph{homogeneous} foliation.
\medskip

\textbf{RFKM-foliations:}  By the work of Ferus, Karcher, and Münzner in \cite{FerusKarcherMuenzner1981} there is an infinity of non-homogeneous closed codimension $1$ foliations  on round spheres $\Sp^k\subset \R^{k+1}$  given by Clifford systems. This construction was generalized by Radeschi in \cite{Radeschi2014} and we present some details here. We recall that solutions to \eqref{Main:Yamabe type} such as  \cite{ClappFernandez2017,FernandezPetean(2020),FernandezPalmasPetean(2020)}  do not apply to the foliations in \cite{Radeschi2014} while our main theorems do apply.

Given a real vector space $V$ of dimension $m$, equipped with a positive definite inner-product $\langle\,,\rangle$ the \emph{Clifford algebra} $\Cli(V)$ is the quotient of the tensor algebra $T(V)$ by the ideal generated by $x\otimes x + y\otimes y - 2\langle x,y\rangle 1$, where $1$ is the unit element in $T(V)$. The vector space $V$ embeds naturally into $\Cli(V)$. A \emph{representation} of a Clifford algebra $\Cli(V)$ is an algebra homomorphism $\rho\colon \Cli(V)\to \End(\R^n)$. A \emph{Clifford system} is the restriction of $\rho$ to $V\subset \Cli(V)$, and we denote by $\R_\rho$ the image $\rho(V)$. Given a Clifford system we can find an inner product on $\R^n$, such that for every $v\in V$, the matrix $\rho(v)$ is a symmetric matrix. We endow the space of all symmetric matrices $\Sym(\R^n)$ with the inner product $\langle A, B\rangle = (1/n)\tr(AB)$. With these choices of inner products on $\R^n$ and $\Sym(\R^n)$ the map $\rho\colon V\to \R_\rho\subset\Sym(\R^n)$ is an isometry onto its image. 

\begin{remark}
Given a Clifford system $\rho\colon \Cli(V)\to \End(\R^n)$, the dimension $n$ is even (see \cite{Radeschi2014}).
\end{remark}

Consider the map $\pi_\rho\colon \Sp^{n-1}\to \R_\rho$ that takes $x\in \Sp^{n-1}$ to the unique element $\pi_\rho(x)\in \R_\rho$ which for all $P\in \R_\rho$ satisfies
\[
    \langle \pi_\rho(x), P \rangle = \langle P(x), x\rangle.
\]
The image of this map is contained in the unit disk of $\R_\rho$, denoted by $\D_\rho$ (see \cite[Proposition~2.4]{Radeschi2014}).

\begin{proposition}[Proposition~2.6 in \cite{Radeschi2014}]\th\label{P: Clifford system induce Singular Riemannian foliations}
The set $\fol_\rho = \{\pi_\rho^{-1}(P)\mid P\in \D_\rho\}$ is a singular Riemannian foliation on the unit sphere $\Sp^{n-1}$  with the round metric.
\end{proposition}
 
We refer to the foliation $\fol_\rho$ in \th\ref{P: Clifford system induce Singular Riemannian foliations} as an RFKM-foliation. As proven in \cite[Section 5]{Radeschi2014}, only a finite number of such foliations are homogeneous.

Given a Clifford system $\rho\colon V \subset \Cli(V) \to \R_\rho$, for the map $f\colon \R_\rho\to \R$ defined as $f(P) = 1-2\Vert P \Vert^2$, the preimages of the composition $f\circ\pi_\rho\colon \Sp^{n-1}\to \R$ induce a singular Riemannian foliation of codimension $1$ on the round sphere $\Sp^{n-1}$. These are the singular Riemannian foliations $\fol_{\mathrm{FKM}}$ described in \cite{FerusKarcherMuenzner1981}. 

In general, given two singular Riemannian foliations $\fol_1$ and $\fol_2$ on a given Riemannian manifold $(M,g)$ we say that $\fol_1\subset\fol_2$ if for any leaf $L\in \fol_1$, there exists a leaf $L'\in \fol_2$, such that $L\subset L'$. With this definition observe that $\fol_\rho\subset \fol_{\mathrm{FKM}}$.

\subsection*{Non-orbit like foliations}

With RFKM-foliations we are able to give examples of closed manifolds with singular Riemannian foliations for which the infinitesimal foliation is not given by a group action. Namely, fix $(L,g_L)$ a closed Riemannian manifold and consider the closed disk $(\D^n,g_0,\fol_0)$ with the Euclidean metric, equipped with a singular Riemannian foliation given by the cone of a non-homogeneous RFKM-foliation $\fol_0$ on the sphere $\Sp^{n-1}$. Now consider the foliated  product $(M_1,g_1,\fol_1) = (L\times\D^n,g_L\oplus g_0,\{L\times \mathcal{L}_v\mid \mathcal{L}_v\in \fol_0\})$. Next we glue two copies of $M_1$ along the boundary via the identity map to obtain a new smooth foliated closed manifold $(M,g,\fol)$ whose orbit space is $\Sp^n$, the double of $\D^n$. For any of the  two singular leaves given by  $L\times\{0\}$, in each copy of $M_1$, the infinitesimal foliation  corresponds to the RFKM-foliation $(\Sp^{n-1},g_0,\fol_0)$. Thus $\fol$ cannot be an orbit-like foliation. This implies that the results in \cite{AlexandrinoCavenaghi2021} do not apply to this foliation, but \th\ref{Theorem:Symmetric Criticality}, and in turn \th\ref{Theorem Main} do apply to this foliation.

\subsection*{Other non-homogeneous singular Riemannian foliations}

It is easy to construct closed non simply-connected manifolds with a singular Riemannian foliation which is not given by a group action. For example as suggested in \cite{Galaz-Garcia2015}, we consider $N$ any closed Riemmannian manifold, and  for $n\geq 5$ a smooth manifold $\mathcal{T}^n$ which is homeomorphic but not diffeomorphic to the $n$-torus, i.e. an \emph{exotic torus}. These manifolds exist due to \cite{Hsiang1969}. Now we consider $M = N\times\mathcal{T}^n$ with any product metric. This produces a  Riemannian foliation $\fol = \{\{p\}\times\mathcal{T}^n\mid p\in N\}$ with leaves of the same dimension, whose leaves are not homogeneous spaces and therefore, this foliation cannot be given by a group action. Example~3.6 in \cite{Galaz-Garcia2015} gives a way to construct other non-homogeneous examples. As a particular case of these examples,  we can see that the Klein bottle has a foliation by circles, which is not given by a circle action. Moreover, in \cite{FarrellWu2018} Farrell and Wu gave examples of Riemannian foliations with all leaves of the same dimension. The leaves of one of theses foliations are the fibers of a fiber bundle over $4$-dimensional manifolds, and moreover these fibers  are exotic tori. Again we point out that the previous results in the literature for finding sign-changing solutions to \eqref{Main:Yamabe type} do not apply for these foliations, but \th\ref{Theorem Main} does.

\section{A Principle of Symmetric Criticality for singular Riemannian foliations}\label{Sec:Symmetryc Criticality}

Let $(M,g,\fol)$ be a singular Riemannian foliation with closed leaves. As before, denote by $H_g^1(M)^\fol$ the closure of the smooth foliated real valued functions $\mathcal{C}^\infty(M)^\fol$ under $\Vert\cdot\Vert_{H^1_g(M)}$ and by $H^\perp_\fol$ the orthogonal complement of $H^1_g(M)^\fol$ with respect to  the $H^1_g(M)$-inner product.  The aim of this section is to prove Theorem \ref{Theorem:Symmetric Criticality}. The proof relies on establishing that $\langle f,u \rangle_{L^2_g(M)} = 0$ for any $f\in H^\perp_\fol$ and $u\in H^1_g(M)^\fol$, and using the explicit formula of the derivative of the energy functional $J$.

To our knowledge, Theorem \ref{Theorem:Symmetric Criticality} fills in a part in \cite[Proof of Lemma 4.1]{Henry2019} where a particular case of Theorem \ref{Theorem:Symmetric Criticality} for foliations given by iso\-pa\-ra\-metric functions was used without proof to obtain certain solutions to the subcritical Yamabe equation.

\bigskip

For a metric space $(X,d)$,
given a  a function $f\colon X\to \R$ we define $\mathrm{Lip}(f)$, the \emph{Lipschitz constant of $f$}, as
\[
    \mathrm{Lip}(f) = \sup\left\{\frac{|f(x)-f(y)|}{d(x,y)}\mid x\neq y \in X\right\}.
\]
Denote by 
\[
\mathrm{LIP}(X) = \Big\{f\colon X\to \R\mid \mathrm{Lip}(f)<\infty\Big\}.
\]

We now let $d_M$ be the distance function on $M$ induced by the Riemannian metric $g$ and consider the subset
\[
\mathrm{Lip}_\fol=\{f\in \mathrm{LIP}(M)\mid f \mbox{ is constant along the leaves of $\fol$}\}.
\]

Denote by $L^2_\fol$ the closure of $\mathrm{Lip}_\fol$ in $L_g^2(M)$, and by $(L^2_\fol)^\perp$ the orthogonal complement of $L^2_\fol$ with respect to  the $L^2_g(M)$-inner product.

Recall that the leaf space of $(M,g, \fol)$ is the space $M^\ast=M/\fol$ equipped with the quotient topology induced by the projection $\pi\colon M\to M^\ast$. For $M$ complete and $\fol$ with closed leaves, this map also endows $M^\ast$ with a metric $d^\ast$ and a measure $\mu^\ast =\pi_\ast dV_g$, which is the pushforward of the Riemannian volume form  $dV_g$ of $(M,g)$ under the quotient map $\pi$. Moreover $(M^\ast, d^\ast, \mu^\ast)$ is a complete separable metric space with a non-trivial locally finite Borel measure (\cite[p.~119]{LytchakThobergsson2010}); that is $(M^\ast, d^\ast, \mu^\ast) $ is a \emph{metric measure space} \cite[p.~65]{HeinonenKoskelaShanmugalingamTyson}. So that $\mathrm{LIP}(M^\ast)$ is dense in $L^2(M^\ast, \mu^\ast)$ (see for example \cite[Theorem~4.2.4]{HeinonenKoskelaShanmugalingamTyson}).

Given a function $f\colon M\to \R$ which is constant along the leaves, we define a function $f^\ast\colon M^\ast\to \R$ as $f^\ast(\pi(x)) = f(x)$. Moreover, given a function $h\colon M^\ast\to \R$, we define a function $\hat{h}\colon M\to \R$ which is constant along the leaves of $\fol$ by $\hat{h}(x) = h(\pi(x))$. For a function $f\colon M\to \R$ constant along the leaves we have $f=\widehat f^\ast$.

Using the fact that $\mu^\ast$ is the pushforward measure,  we have well defined functions  $\ast\colon  L_\fol^2\to L^2(M^\ast,\mu^\ast)$ and $\ \widehat{}\,\colon L^2(M^\ast,\mu^\ast)\to L_\fol^2$. Indeed, take $f\in L^2_\fol$. Then it holds that 
\[
\int_{M^\ast} (f^\ast)^2\, \mu^\ast = \int_{M} (f^\ast\circ\pi)^2\, dV_g = \int_M f^2\, dV_g<\infty.
\]
On the other hand, for $h\in L^2(M^\ast,\mu^\ast)$ we have that 
\[
\int_{M} \hat{h}^2\, dV_g = \int_M (h\circ\pi)^2\, dV_g = \int_{M^\ast} h^2\, \mu^\ast<\infty.
\]

\begin{lemma}\th\label{lem-HsubsetL}
For $(M,g, \fol)$ a singular Riemannian foliation with closed leaves, it holds true that $H^1_g(M)^\fol\subset L^2_\fol$.
\end{lemma}

\begin{proof}
Given $u\in H^1_g(M)^\fol$, since  $\mathrm{LIP}(M^\ast)$ is dense in $L^2(M^\ast, \mu^\ast)$, there exists a sequence $\{h_n\}\subset L^2(M^\ast,\mu^\ast)$ of Lipschitz bounded functions converging to $u^\ast$ in $L^2(M^\ast,\mu^\ast)$.
Applying a change of variable we get that $\hat{h}_n\to \widehat u^\ast=u$ in $L_g^2(M)$. To prove this,  fix $\varepsilon >0$.  Then there exists $N\in \mathbb{N}$, such that for $n\geqslant N$ we have:
\begin{linenomath}
\begin{align*}
    \int_M |\hat{h}_n-u|^2\,dV_g &= \int_M |\hat{h}_n-\widehat{u^\ast}|^2\,dV_g
     = \int_M |h_n\circ\pi-u^\ast\circ\pi|^2\,dV_g\\
    &= \int_{M^\ast} |h_n-u^\ast|^2\,\pi_\ast dV_g
    = \int_{M^\ast} |h_n-u^\ast|^2\,\mu^\ast\\
    &<\varepsilon.
\end{align*}
\end{linenomath}
Thus, $\hat{h}_n\to u$ in $L_g^2(M)$. Finally note that by definition $\hat h_n= h_n \circ \pi$ is the composition of Lipschitz functions, $\pi$ being a $1$-Lipschitz function. Hence,  $\hat{h}_n \in \mathrm{LIP}_\fol$. Thus, we conclude that $u \in L^2_\fol$ as desired.
\end{proof}

\begin{lemma}\th\label{L: H1(M) fol perp in L2(M) fol perp}
For $(M,g, \fol)$ a singular Riemannian foliation with closed leaves, it holds true that $H_\fol^\perp\subset (L^2_\fol)^\perp$.
\end{lemma}

\begin{proof}
We first show that $H^\perp_\fol\cap L^2_\fol\subset L^2_\fol$ consists only of the equivalent class of the zero function. Take $u\in H^\perp_\fol\cap L^2_\fol\subset L^2_\fol$. Since  $u\in L^2_\fol$ there exists a sequence of Lipschitz functions which are constant along the leaves, $\{f_n\}\subset \mathrm{Lip}_\fol$, converging to $u$ with respect to the $L^2_g(M)$-norm. Then $\{f_n^\ast\}$ is a Cauchy sequence in $L^2(M^\ast,\mu^\ast)$; fix $\varepsilon>0$, there exists $N\in \N$ such that for all $m,n\geqslant N$ it holds:
\begin{linenomath}
\begin{align*}
    \varepsilon &> \int_M |f_n-f_m|^2\,dV_g 
     = \int_M |\widehat{f_n^\ast}-\widehat{f_m^\ast}|^2\,dV_g\\
    &= \int_M |f_n^\ast\circ\pi-f_m^\ast\circ\pi|^2\,dV_g
     = \int_{M^\ast} |f_n^\ast-f_m^\ast|^2\, \pi_\ast dV_g\\
     &= \int_{M^\ast} |f_n^\ast-f_m^\ast|^2\, \mu^\ast.
\end{align*}
\end{linenomath}
Since $L^2(M^\ast,\mu^\ast)$ is complete there is a limit $h\in L^2(M^\ast,\mu^\ast)$  of $\{f_n^\ast\}$. As in the proof of the previous lemma,  $\widehat{f_n^\ast}=f_n$  converges to $\hat{h}$ in $L_g^2(M)$. Thus by uniqueness of the limit,  $u$ is equal to $\hat{h}$ $\mu$-a.e., i.e. $u$ is constant along the leaves of $\fol$ up to zero measure. This implies that  the class of $u$ in $H^1_g(M)$ is an element in $H^1_g(M)^\fol$ and since $u \in H^\perp_\fol$ by hypothesis, we conclude that $u$ corresponds to the class of the zero function in $H^1_g(M)$ (recall that $H^1_g(M)= H^1_g(M)^\fol\oplus H^\perp_\fol$). Thus $H^\perp_\fol \cap L^2_\fol = \{0\}$. 

Finally, since $L^2_\fol \cap (L^2_\fol)^\perp$ consists only of the equivalence class of the zero function, by the first paragraph we conclude that $H_\fol^\perp\subset (L^2_\fol)^\perp$.
\end{proof}

\begin{lemma}\label{lem:H1Ortho}
Let $(M,g,\fol)$ be a singular Riemannian foliation with closed leaves. If $f\in H^\perp_\fol$ and $u\in H^1_g(M)^\fol$ then $\langle f,u \rangle_{L^2_g(M)} = 0$.
\end{lemma}

\begin{proof}
By Lemma \ref{lem-HsubsetL} and Lemma \ref{L: H1(M) fol perp in L2(M) fol perp}, we have that $f\in (L^2_\fol)^\perp$ and $u\in L^2_\fol$. The lemma then follows.
\end{proof}

We are ready to prove the Principle of Symmetric Criticality, that is, 
that critical points of $J$ restricted to $H_g^1(M)^\fol$ are critical points of $J$.
	
 \begin{proof}[Proof of Theorem \ref{Theorem:Symmetric Criticality}]
 Write any $v\in H^1_g(M)$ as $v=v_1+v_2$ with $v_1\in H_g^1(M)^\fol$ and $v_2\in H_\fol^\perp$. 
  Since $u\in H_g^1(M)^\fol$ is a critical point of $J$ restricted to the space $H^1_g(M)^\fol$, we have that $J'(u)v_1=0$. Hence, 
\begin{linenomath}
 \begin{align*}
 J'(u)v &=J'(u)v_1+J'(u)v_2=J'(u)v_2\\
 &= \langle u,v_2 \rangle_b - \int_M c\vert u\vert^{p-2}uv_2 dV_g, 
\end{align*}
\end{linenomath}
where in the last part we  used the expression \eqref{Eq:Derivative J} for the derivative of $J$. By Lemma \ref{lem:H1Ortho} both terms above equal zero. To see this is true for the first term, by hypothesis we have $\langle u,v_2 \rangle_{H^1_g(M)}=0$ and by Lemma \ref{lem:H1Ortho} we have that $\int_M uv_2 \ dV_g=0$. This and $b$ being an $\fol$-invariant function implies that $\langle u,v_2 \rangle_b=0$. For the second term, since $c$ is $\fol$-invariant we have $c\vert u\vert^{p-2}u\in L_{\fol}^2$ and, we know that $v_2 \in H_\fol^\perp \subset (L^2_\fol)^\perp$. Thus, we can apply again Lemma \ref{lem:H1Ortho} to conclude the claim. 
\end{proof}

\section{Compactness}\label{Sec:Compactness}

In this section we first prove \th\ref{Theorem:Sobolev Embedding} which is a Sobolev embedding theorem for singular Riemannian foliations that generalizes the classical result by Hebey and Vaugon in \cite{HebeyVaugon1997} and Lemma 6.1 in \cite{Henry2019}. Then we apply \th\ref{Theorem:Sobolev Embedding} to show \th\ref{Theorem:Compactness},
which says that $J$ satisfies the $(PS)^\fol_\tau$-condition for every $\tau\in\R$. 

\bigskip

For any $s\geq 1$, consider the Sobolev space $H_g^{1,s}(M)$, which is the closure of $\mathcal{C}^\infty(M)$ with respect to the norm
\[
\Vert u\Vert_{H_g^{1,s}(M)}:=\left(\int_M \vert \nabla_g u\vert^s_g + \vert u\vert^s dV_g\right)^{1/s}.
\]
As before, let $H_g^{1,s}(M)^{\fol}$ be the closure of $\mathcal{C}^\infty(M)^\fol$ under the norm above, so that it is a Banach space. Observe that $H^1_g(M)$ is just 
$H^{1,2}_g(M)$.

Recall that  $\kappa_\fol $ is the smallest dimension of the leaves of $\fol$. We now state the Sobolev embedding theorem needed for the proof of Theorem \ref{Theorem:Compactness}.

\begin{theorem}\th\label{Theorem:Sobolev Embedding}
	Let $\mathcal{F}$ be a singular Riemannian foliation on the $m$-dimensional Riemannain manifold $(M,g)$	and let $s\geq 1$. If any of the following conditions hold,
\begin{equation}\label{C1}
s\geq m-\kappa_\fol \quad \text{and} \quad p\geq 1 \tag{C1}
\end{equation}
or
	\begin{equation}\label{C2}
	s<m-\kappa_\fol  \quad \text{and} \quad p\in\Big[1,\frac{s(m-\kappa_\fol)}{m-\kappa_\fol-s} \Big] \tag{C2},
\end{equation}
then the inclusion map $H_{g}^{1,s}(M)^\fol\hookrightarrow L^{p}_g(M)$ is continuous. For  either $p\geqslant 1$ in \eqref{C1} or $1\leqslant p < \frac{s(m-\kappa_\fol)}{m-\kappa_\fol-s}$ in \eqref{C2}, the map is also compact.

\end{theorem}

The proof of this is virtually the same as the proof of Lemma 6.1 in \cite{Henry2019} or the proof of the Main Lemma in \cite{HebeyVaugon1997}. But we present it below for the convenience of the reader. 
To show the continuity of the inclusion, the idea is to apply \th\ref{Lemma:Charts of foliation around a point}
to any $q\in M$, to get a chart of the form $(W_q,\varphi_q)$.   As $M$ is compact, we can cover it by open sets of the form $W_q$. Then one passes the information of any function in $H_{g}^{1,s}(M)^\fol$ to functions defined on sets $V_q$, where $\varphi_q(W_q)=U_q \times V_q$ as in \th\ref{Lemma:Charts of foliation around a point}, reducing the dimension of $M$ to the dimension of $V_q$, which is $m-k_q$. This allows us to apply the subcritical Sobolev embedding theorem for the sets $V_q$. To prove the compactness of the map,  we also apply  \th\ref{T: Slice Theorem} and  \th\ref{L: Existence of foliated Partition of unity} in order to have a covering of $M$ consisting of tubular neighborhoods and finalize the proof by applying the Rellich-Kondrachov compactness theorem.

\bigskip
 We start by fixing some notation and studying 
the behavior of functions in $H_{g}^{1,s}(M)^\fol$.  Fix $q\in M$, and let $\varphi\colon W \to U\times V$ be a chart around $q$ given by \th\ref{Lemma:Charts of foliation around a point}. For any function  $u\in H_g^{1,s}(M)^\fol$, let $\hat{u}:=u\circ\varphi^{-1}\colon U\times V\to \R$. By the properties of the coordinate chart, for any $x,x'\in U$ and any $y\in V$, we have that $\hat{u}(x,y)=\hat{u}(x',y)$, so we have a well defined function $u_{V}\colon V\to \R$ given by $u_{V}(y):=\hat{u}(x,y)$. Denote by $\nabla_g u$ the gradient of $u$, for $z=(x,y)\in U\times V$ then we denote by $\nabla \hat{u}= (\nabla_x \hat{u},\nabla_y \hat{u})$ the gradient of $\hat{u}$, and by $\nabla_y u_V$ the gradient of $u_V$.

\begin{lemma}\th\label{lem-equivGradients}
Consider a fixed $q\in M$, and $W$ as above. For any open subsets
$U' \subset\subset U$ and $V' \subset\subset V$ and any
$u\in \mathcal{C}^\infty(M)^\fol$
 there exists a constant $C=C(W')$, where $W'= \varphi^{-1}(U' \times V')$, such that 
\[
C^{-1}\vert \nabla_y u_V\vert^2\leq \vert\nabla_g u\vert_g^2\leq C\vert \nabla_y u_V\vert^2
\]
holds over $W'$.
\end{lemma}

\begin{proof}
Since $u\in \mathcal{C}^\infty(M)^\fol$ we have $\frac{\partial \hat{u}}{\partial x^i}(x,y)=0$. Hence,  $\nabla \hat{u}=(\nabla_x \hat{u},\nabla_y \hat{u})=(0,\nabla_y u_V)$ and thus $\vert \nabla_g u\vert^2_g =\vert \nabla \hat{u}\vert^2=\vert \nabla_y u_V\vert^2 $.

Writing the metric $g$ in coordinates as $g=\sum_{i,j=1}^m g_{ij}dx^i\otimes dx^j$, at each $q'\in W$ the matrix $(g_{ij}(q'))$ is definite positive and symmetric, hence it is diagonalizable and with positive eigenvalues $\lambda_i(q')$. Then, at each $q'$ we have that
\begin{linenomath}
\begin{align*}
\min\{\lambda_i(q')\}\sum_{i=1}^m\left(\frac{\partial \hat{u}}{\partial y^i}(q')\right)^2 &\leq \sum_{i,j=1}^mg_{ij}(q')\left(\frac{\partial \hat{u}}{\partial y^i}(q')\right)\left(\frac{\partial \hat{u}}{\partial y^j}(q')\right)\\
&\leq \max\{\lambda_i(q')\}\sum_{i=1}^m\left(\frac{\partial \hat{u}}{\partial y^i}(q')\right)^2.
\end{align*}
\end{linenomath}
Since the functions $\lambda_i$ are continuous, and $W' \subset \subset W$ then 
\[
A:=\min_{q'\in \overline{W'}}\min\{\lambda_i(q')\},\quad B:=\max_{q'\in \overline{W'}}\max\{\lambda_i(q')\}>0.\] 
If $C>0$ is such that $C^{-1}<A\leq B<C$, substituting everything in the previous inequality, and using that $\vert \nabla \hat{u}\vert^2=\vert \nabla_y u_V\vert^2$ we conclude.
\end{proof}

\begin{lemma}\th\label{Lemma:Equivalent Norms}
For $q\in M$, $W$ and any $W'$ as in the previous lemma, there exists a constant $C=C(W')$, such that, for any $s\geq 1$ and any function $u\in\mathcal{C}^\infty(M)^\fol$
\[
C^{-1}\int_{V'} \vert u_{V}\vert^s dy\leq  \int_{W'} \vert u\vert^s dV_g \leq C \int_{V'} \vert u_V\vert^s dy
\]
and
\[
C^{-1}\int_{V'} \vert \nabla_y u_V\vert^s dy\leq  \int_{W'} \vert \nabla_g u\vert_g^s dV_g \leq C \int_{V'} \vert \nabla_y u_V\vert^s dy.
\]
In particular, there exists $C=C(W')>0$ such that for any $u\in H^{1,s}_g(M)^\fol$,
\begin{linenomath}
\begin{align}\begin{split}
  &C^{-1}\vert u\vert_{L_g^s(W')}\leq \vert u_V\vert_{L^s(V')} \leq C\vert u\vert_{L_g^s(W')},\\\label{Eq:Local Lebesgues Norm Equivalence}
\mbox{ and } &\\ 
 &C^{-1}\vert \nabla_g u\vert_{L_g^s(W')}\leq \vert \nabla u_V\vert_{L^s(V')} \leq C\vert \nabla_g u\vert_{L_g^s(W')}.
\end{split}
\end{align}
\end{linenomath}
\end{lemma}

\begin{proof}
For the first inequality, the function $\sqrt{\vert g\vert}\colon \overline{W'}\to \R$ is continuous and attains its maximum and its minimum, say $0<A:=\min \sqrt{\vert g\vert}$ and $0<B=\max \sqrt{\vert g\vert}$. 

First we have, 
\begin{linenomath}
\begin{align*}
\int_{W'}\vert u\vert^s dV_g &= \int_{U'\times V'}\sqrt{\vert g\vert}\vert \hat{u}\vert^s dx dy 
\leq B\int_{U'\times V'}\vert u_V\vert^s dx dy\\
&= B\int_{U'}dx\int_{V'}\vert u_V\vert^s dy
=C_1\int_{V'}\vert u_V\vert^s dy.
\end{align*}
\end{linenomath}
Second we have
\begin{linenomath}
\begin{align*}
\int_{V'}\vert u_V\vert^s dy &= A^{-1}\int_{V'} A\vert u_V\vert^s dy
= C_2 \int_{U'\times V'} A\vert u\vert^s dxdy\\
&\leq C_2 \int_{U'\times V'}\sqrt{\vert g\vert}\vert u\vert^s dxdy
= C_2\int_{W'}\vert u\vert^s dV_g.
\end{align*}
\end{linenomath}

To obtain  the gradient estimate, we use \th\ref{lem-equivGradients}:
\begin{linenomath}
\begin{align*}
\int_{W'}\vert\nabla_g u\vert_g^s dV_g&=\int_{U'\times V'}\sqrt{\vert g\vert}\vert\nabla_g u\vert_g^s dxdy
\leq B_2 \int_{U'\times V'}\vert \nabla_yu_V\vert^{s} dxdy\\
&= B_2 \int_{U'}dx\int_{V'}\vert \nabla_yu_V\vert^{s}dy
=C_3\int_{V'}\vert \nabla_yu_V\vert^{s}dy.
\end{align*}
\end{linenomath}
And again:
\begin{linenomath}
\begin{align*}
&\int_{V'}\vert \nabla_yu_V\vert^{s}dy =A^{-1}\int_{V'}A\vert \nabla_y u_V\vert^{s}dxdy
= C_4\int_{U'\times V'}A\vert \nabla_y u_V\vert^{s}dxdy\\
&\leq C_4\int_{U'\times V'}\sqrt{\vert g\vert}\vert \nabla_y u_V\vert^{s} dxdy
\leq C_5 \int_{U'\times V'}\sqrt{\vert g\vert}\vert \nabla_g u\vert_g^{s}dxdy\\
&=C_5\int_{W'}\vert \nabla_g u\vert_g^s dV_g.
\end{align*}
\end{linenomath}

Taking $C$ to be any real constant such that 
\[
C^{-1} \leq \min\{C_1^{-1/s},C_3^{-1/s}\} \mbox{ and } \max\{C_2^{1/s},C_5^{1/s}\}\leq C
\]
we obtain the inequalities \eqref{Eq:Local Lebesgues Norm Equivalence} for any $u\in\mathcal{C}^\infty(M)^\fol$. Since $\mathcal{C}^\infty(M)^\fol$ is dense in $H^{1,s}_g(M)^\fol$ and in $L_g^s(M)^\fol$, we conclude that  \eqref{Eq:Local Lebesgues Norm Equivalence} holds.
\end{proof}

We are now ready to prove the Sobolev embedding theorem.

\begin{proof}[Proof of Theorem \ref{Theorem:Sobolev Embedding}]
For any $q\in M$, take a chart around $q$, $(W_q,\varphi_q)$, as in  \th\ref{Lemma:Charts of foliation around a point}. 
Take a further open subset $W'_q \subset \subset W_q$ around $q$ as in Lemma \ref{Lemma:Equivalent Norms}. As $M$ is compact, there exist a finite number of points $q_1,\ldots,q_\ell\in M$ such that the charts $(W_i,\varphi_i):=(W'_{q_i},\varphi_{q_i}|_{W'_{q_i}})$ cover $M$. Set $k_i:=k_{q_i}$ and note that $\varphi_i(W_i)=U_i\times V_i$ for some open sets.

To prove the continuity of the inclusion map, take $u\in H_g^{1,s}(M)^\fol$ and consider the function $u_i = u_{V_i}: V_i \to \mathbb R$ as defined in this section, i.e. $u_{V_i}(y):=u\circ\varphi_{i}^{-1}(x,y)$, where $x$ is an  arbitrary element in $U_i$. Then Lemma \ref{Lemma:Equivalent Norms} implies that $u_i\in H^{1,s}(V_i)$. Using  \eqref{Eq:Local Lebesgues Norm Equivalence} twice together with the Sobolev embedding theorems for $V_i\subset \R^{m-k_i}$ \cite[Theorem 5.4]{Adams}, we have the existence of a constant $C_i$ depending on $W_i$ such that
\begin{linenomath}
\begin{align*}
\vert u\vert_{L^{p_i}_g(W_i)}&\leq C_i\vert u_i\vert_{L^{p_i}(V_i)}\\
&\leq C_i\left(\vert \nabla_y u_i\vert_{L^s(V_i)}+\vert u_i\vert_{L^s(V_i)}\right)\\
&\leq C_i\left(\vert \nabla_g u\vert_{L^s_g(W_i)}+\vert u\vert_{L^s_g(W_i)}\right)
\end{align*}
\end{linenomath}
if either $s \geq m-k_i$ and $p_i\geq 1$,  
or $s < m-k_i$ and $1 \leq p_i\leq s(m-k_i)/(m-k_i-s)$.

We now prove that  for every index $i = 1,\ldots,\ell$, 
either $s \geq  m-k_i$ or $s < m-k_i$ and $p\leq s(m-k_i)/(m-k_i-s)$ holds. 
If \eqref{C1} holds, then by the definition of $\kappa_\fol$ we get $s \geq m-k_i$ for every $i=1,\ldots,\ell$.  If  \eqref{C2} holds, fix $i$. Then
either $s \geq m-k_i$ or $s< m-k_i$. Assume that  $s < m-k_i$. Then the fact $s\geqslant 0$ and the definition of $\kappa_\fol$ yields the inequality $s(m-\kappa_\fol)/(m-\kappa_\fol-s)\leq  s(m-k_i)/(m-k_i-s)$. 

Therefore, by taking $C\geqslant C_i$ for all $i$, 
\begin{linenomath}
\begin{align*}
\vert u\vert_{L_g^p(M)} &\leq \sum_{i=1}^\ell \vert u\vert_{L_g^p(W_i)}
\leq C  \sum_{i=1}^\ell \vert \nabla_g u\vert_{L_g^s(W_i)} + \vert u\vert_{L_g^s(W_i)}\\
&\leq C\sum_{i=1}^\ell \Vert u\Vert_{H_g^{1,s}(M)}
\leq C\ell \Vert u\Vert_{H_g^{1,s}(M)}.
\end{align*}
\end{linenomath}
Hence the embedding $H_g^{1,s}(M)^\fol\hookrightarrow L^{p}_g(M)$ is continuous. 

Now we prove the compactness of the embedding. Without loss of generality assume that each open set $W_i = U_i\times V_i$ is a trivialization of the disk bundle $P_i\times_{\Hol(L_i)} \D_i^\perp \to L_i$ given by \th\ref{T: Slice Theorem} centered at some $p_i\in M$. We can identify each element of the finite open cover $ \mathcal{A} = \{\mathrm{Tub}^{\varepsilon(p_i)}(L_{p_i})\}_{i=1}^N$ of $M$, with $P_i\times_{\Hol(L_i)} \D_i^\perp$. By \th\ref{L: Existence of foliated Partition of unity} there exists a smooth foliated partition of unity $\{\phi_i\}$ subordinated to $\mathcal{A}$. For each index $i$, due to the identification we may assume that $\phi_i$ is defined over $P_i\times_{\Hol(L_i)} \D_i^\perp$. 

Let $(u_j)_{j\in\mathbb{N}}$ be a bounded sequence in $H^{1,s}_{g}(M)^\fol$ and define $u_{ji}(y):=(\phi_i\circ\varphi_i^{-1})(u_j\circ\varphi_i^{-1})(x,y)$, for arbitrary $x\in U_i$. Observe $u_{ji}$ is well defined and is compactly supported in $V_i$ and that the sequence $(u_{ji})_{j\in\mathbb{N}} $ is bounded in $H^{1,s}(V_i)$. Then, if either $s\geq m-\kappa_\fol$ or if $s<m-\kappa_\fol$ and $1\leq p<\frac{s(m-\kappa_\fol)}{m-\kappa_\fol-s}$, we have that $\frac{1}{p}>\frac{1}{s}-\frac{1}{m-k_i}$ for every $i=1,\ldots \ell$. Hence, the embedding $H_0^{1,s}(V_i)\hookrightarrow L^{p}(V_i)$ is compact by the Rellich-Kondrachov theorem \cite[Theorem 6.5]{Adams} and for each $i=1,\ldots,\ell$, there is a subsequence of $(u_{ji})_{j\in\mathbb{N}}$, which we denote in the same way, which is a Cauchy sequence in $L^p(V_i)$. Inequalities \eqref{Eq:Local Lebesgues Norm Equivalence} and the fact that $\{\phi_i\}_{i=1}^\ell$  are bounded yield that $(u_j)_{j\in\mathbb{N}}$ has a Cauchy subsequence in $L^p_g(M)$ and, hence, this subsequence converges in this space.
\end{proof}

We now show that $J$ satisfies the $(PS)^\fol_\tau$-condition for every $\tau\in\R$. 

\begin{proof}[Proof of Theorem \ref{Theorem:Compactness}] The proof is standard and consists in showing that any $(PS)^\fol_\tau$-sequence for $J$, $(v_n)_{n\in \N}\subset H_g^1(M)^\fol$, is bounded in $H_g^1(M)$, so there exists $v\in H_g^1(M)$ such that up to a subsequence, $(v_{n})_{n\in \N}$ converges weakly to $v$ in $H_g^1(M)^\fol$. The next step is to apply Theorem \ref{Theorem:Sobolev Embedding} to show that   $\Vert v_{n}\Vert_{H^1_g(M)} =\Vert v_{n}\Vert_{H^1_g(M)}\to\Vert v\Vert_{H^1_g(M)}$, which implies that $v_n\to v$ strongly in $H_g^1(M)^\fol$. We provide more details below. 
 
Let $(v_n)_{n\in \N}\subset H_g^1(M)^\fol$ 
be a $(PS)^\fol_\tau$-sequence for $J$ in $H_g^1(M)^\fol$, i.e. $J(v_n)\to \tau$ and $J'(v_n)\to 0$ in $(H_g^1(M)^\fol)^{\ast}$. We first uniformly bound $\Vert v_n\Vert_b$ in terms of $J(v_n)$ and $J'(v_n)v_n$.

From the definition of $J$ we get 
	\[
	J'(v_n)v_n =\Vert v_n \Vert^2_b - \vert v_n\vert^p_{c,p}.
	\]
	Therefore, we have
\begin{linenomath}
\begin{align*}
J(v_n) - \tfrac{1}{p}J'(v_n)v_n &= \tfrac{1}{2}\Vert v_n\Vert^2_b -\tfrac{1}{p}\vert v_n\vert^p_{c,p} -\tfrac{1}{p}\Vert v_n\Vert^2_b +\tfrac{1}{p}\vert v_n\vert^p_{c,p}\\
&=\tfrac{p-2}{2 p}\Vert v_n\Vert^2_b.
\end{align*}
\end{linenomath}

On the other hand, since $J'(v_n)\to 0$ in $(H_g^1(M)^\fol)^{\ast}$, there exists a positive constant $C_1$  such that for large $n\in \N$ the following holds
\[
	\vert J'(v_n)v_n\vert \leq \Vert J'(v_n)\Vert_{H_g^1(M)^{\ast}}\Vert v_n\Vert_b \leq C_1\Vert v_n \Vert_b.
\]
Since $J(v_n)\to \tau$, there exists a positive constant $C_2$ such that for large $n\in \mathbb N$ the following holds
\[
	\vert J(v_n) \vert \leq C_2.
\]	
Therefore for large $n\in \N$ we obtain that 
\begin{linenomath}
\begin{align*}
    \tfrac{p-2}{2p}\Vert v_n \Vert^2_b &= J(v_n) - \tfrac{1}{p} (J)'(v_n)v_n 
	\leqslant \big\vert  J(v_n) - \tfrac{1}{p} (J)'(v_n)v_n  \big\vert \\
	&\leqslant \vert J(v_n)\vert + \tfrac{1}{p} \vert(J)'(v_n)v_n \vert
	\leqslant C_2 + \tfrac{C_1}{p}\Vert v_n\Vert_b. 
\end{align*}
\end{linenomath}
This, and the fact that the norm $\Vert\cdot\Vert_b$ is equivalent to the standard norm of $H^1_g(M)$, imply that $(v_n)_{n\in \N}$ is bounded in $H^1_g(M)$.

Since $(v_n)_{n\in \N}$ is bounded in $H_g^1(M)^\fol$, there exists $v\in H_g^1(M)$ such that up to a subsequence, also denoted by $(v_{n})_{n\in \N}$, it converges weakly to $v$ in $H_g^1(M)$. As $H_g^1(M)^\fol$ is weakly closed in $H_g^1(M)$, $v$ is $\fol$-invariant. To show that $v_n\to v$ strongly in $H_g^1(M)^\fol$, it suffices to show that $\Vert v_{n}\Vert_{H^1_g(M)}\to\Vert v\Vert_{H^1_g(M)}$. 
 
 As $v_{n}\rightharpoonup v $ weakly in $H_g^1(M)^\fol$, we have that 
 \[
 \langle v_n, v \rangle_{H^1_g(M)} =
 \Vert v\Vert^2_{H^1_g(M)} + o(1).
 \] 
 It follows that
 \begin{linenomath}
\begin{align*}
	J'(v_n)(v_{n}-v)  = & \langle v_{n} , v_{n}-v\rangle_{H^1_g(M)} - \int_{M} c\,\vert v_{n}\vert^{p-2}v_{n}(v_{n} - v)dV_g\\
		= & \Vert v_{n}\Vert^2_{H^1_g(M)} - \Vert v\Vert^2_{H^1_g(M)} + o(1)  \\
		&\qquad - \int_{M} c\,\vert v_{n}\vert^{p-2}v_{n}(v_{n} -  v)dV_g,
\end{align*}
\end{linenomath}
and so, $\Vert v_{n}\Vert_{H^1_g(M)} \to \Vert v\Vert_{H^1_g(M)}$ as $n\to\infty$ if both quantities to the left and to the right of the previous expression converge to zero. 

We first deal with the term to the left. 
Since $(v_{n})_{n\in \N}$ is bounded, then  $(v_{n}-v)_{n\in \N}$ is also bounded, say by a positive constant $C_3$, and thus we get,
\begin{linenomath}
\begin{align}
\begin{split}
	\left\vert J'(v_n)(v_{n}- v) \right\vert &\leq \Vert J'(v_n)\Vert_{(H_g^1(M)^\fol)^{\ast}}\Vert v_{n} - v \Vert_{H_g^1(M)}\\
	&\leq C_3\Vert J'(v_n)\Vert_{(H_g^1(M)^\fol)^{\ast}} \rightarrow 0.\label{Lemma:PS condition J:Eq 1}
\end{split}
\end{align}
\end{linenomath}

For the term to the right, we proceed as follows. Observe that  by \th\ref{Theorem:Sobolev Embedding} when $2\geqslant m-\kappa_\fol$,  then for any $2< p < 2^\ast_m$ the map $H^{1}_g(M)^\fol\to L^p_g(M)$ is compact. In the case when $m-\kappa_\fol>2$ due to the fact that $\kappa_\fol\geqslant 1$ we get that $2^\ast_m < 2(m-\kappa_\fol)/(m-\kappa_\fol-2)$, and thus for $1\leqslant p< 2(m-\kappa_\fol)/(m-\kappa_\fol-2)$ the map $H^{1}_g(M)^\fol\to L^p_g(M)$ is also compact by \th\ref{Theorem:Sobolev Embedding}.

Hence, up to a subsequence, we have that $(v_n)_{n\in \N}$ converges weakly to $v$ in $L_g^{p}(M)$. Using this, Hölder's inequality for $\tfrac{p-1}{p}+\tfrac{1}{p}=1$, Sobolev inequality for the embedding $H_g^1(M)\hookrightarrow L_g^{p}(M)$, the fact that $(v_n)_{n\in \N}$ is bounded in $H_g^1(M)^\fol$ and $c\leq \Vert c\Vert_{L^{\infty}_g(M)}$, we get that
\begin{linenomath}
\begin{align}\label{Lemma:PS condition J:Eq 2}
\begin{split}
\Big\vert  \int_{M} c\, \vert v_{n}&\vert^{p-2} v_{n}( v_{n} - v ) dV_g  \Big\vert 
\leq  \int_{M} c\,\vert v_{n}\vert^{p-1} \vert v_{n} - v \vert dV_g\\
&\leq \Vert c\Vert_{L^\infty_g(M)} \vert v_{n}\vert_{L^p_g(M)}^{p-1} \vert v_{n} - v \vert_{L^{p}_g(M)} \\
&\leq  \Vert c\Vert_{L^\infty_g(M)} \vert v_{n} - v \vert_{L^p_g(M)} \left( C_1^{p-1}\Vert v_{n} \Vert^{p-1}_{H^1_g(M)}\right)\\
&\leq C \vert v_{n} - v \vert_{L^p_g(M)}\to 0, 
\end{split}
\end{align} 
\end{linenomath}
where $C$ denotes some positive constant.

Finally, from the expression we have for $J'(v_n)(v_{n}-v)$, \eqref{Lemma:PS condition J:Eq 1} and \eqref{Lemma:PS condition J:Eq 2}, it follows that $\Vert v_{n}\Vert_{H^1_g(M)} \to \Vert v\Vert_{H^1_g(M)}$ as $n\to\infty$, as we wanted to prove.
\end{proof}

\section{Variational Principle}\label{Section:Variational Principle}

This section is devoted to the proof of Theorem
\ref{Theorem:Variational Principle}. We adapt the proof of \cite[Theorem~2.3]{ClappFernandez2017} (see also  \cite{ClappPacella2008,ClappWeth2005}) to our context using Lemma \ref{L: Aux PDE's} below.

\bigskip
In what follows, $\vert\cdot\vert_{L^{\infty}_g(M)}$ denotes  the usual $L^\infty$ norm. Recall that $c,b\in \mathcal{C}^\infty(M)$ are foliated functions with $c>0$, and that $\vert u\vert_{c,p}$ as defined in \eqref{eq-norm-cp} is a norm in $L^p_g(M)$ equivalent to the standard norm of $L^p_g(M)$. 

For any $\theta>\max\{ 1,\mu,\vert b\vert_\infty\}$, where $\mu$ was defined in \eqref{eq:coercivity},  we have a well defined interior product
\[
\langle u,v\rangle_{\theta}:=\int_M \langle \nabla_g u,\nabla_g v\rangle_g + \theta\, uv\; dV_g,\quad u,v\in H_g^1(M),
\]
which induces a norm, $\Vert\cdot\Vert_\theta$, that is equivalent to the standard norm of $H_g^1(M)$.

Recall that we are studying the functional $J\colon H^1_g(M)\to \mathbb{R}$, given by 
\begin{linenomath}
\[
J(u) = \tfrac{1}{2}\Vert u\Vert^2_b- \tfrac{1}{p}\vert u\vert_{c,p}^p,
\]
\end{linenomath}
restricted to the space $H_g^1(M)^\fol$, and where $\Vert \cdot \Vert_b$ is the norm induced by the inner product defined in \eqref{eq-innerprod-b}. For $p\in(2,2^\ast_m]$, this functional is of class $C^2$ and its gradient with respect to the interior product $\langle \cdot,\cdot\rangle_\theta$ defines a $C^1$ function $\nabla J\colon H_g^1(M)\to H_g^1(M)$. We next show that actually $\nabla J\colon H_g^1(M)^\fol\to H_g^1(M)^\fol$. To do so, we decompose the 
gradient of $J$ at $u\in H^1_g(M)$ as follows.

First, note that with respect to $\langle\cdot,\cdot\rangle_\theta$,
 $\nabla J(u)$, is the vector in $H^1_g(M)$ which satisfies:
\begin{linenomath}
\begin{align*}
     \langle \nabla J(u),v\rangle_\theta &=   J'(u)v\\
    &= \langle u,v\rangle_\theta -\int_M (\theta-b)uv dV_g -\int_M c\vert u\vert^{p-2}uvdV_g.
\end{align*}
\end{linenomath}
for any $v\in H^1_g(M)$ (see \ref{Eq:Derivative J}). Then, we need to apply the following result that we prove in Section~\ref{sub-appendix}.

\begin{lemma}\th\label{L: Aux PDE's}
Let $f\in L_g^2(M)$  and $\theta$ a nonnegative constant. Then the non-homogeneous problem 
\begin{linenomath}
\begin{align*}
-\Delta v +\theta v = f,\quad\text{on }M,
\end{align*}
\end{linenomath}
admits a unique solution in $H^1_g(M)$. Moreover, if $f$ is $\fol$-invariant, then the solution lies in $H_g^1(M)^\fol$.
\end{lemma}

Thus, for each $u\in H_g^1(M)$, by applying \th\ref{L: Aux PDE's} to $f$ equal to $(\theta-b)u$ and $c\vert u\vert^{p-2}u$, which are elements of $L_g^2(M)$, there exist unique solutions $Lu$ and $Gu$ to the non-ho\-mo\-ge\-ne\-ous linear problems
\begin{linenomath}
\begin{align*}
-\Delta_{g}(Lu) + \theta Lu &  =(\theta - b) u,\quad\text{on }M\\
-\Delta_{g}(Gu) + \theta Gu &  =c \left\vert
u\right\vert ^{p-2}u,\quad\text{on }M. 
\end{align*}
\end{linenomath}
Note that these solutions are uniquely determined  by the relations
\begin{linenomath}
\begin{align}
\begin{split}
\langle Lu,v\rangle_\theta & = \int_M(\theta - b) uv\; dV_g,\quad \label{Eq:Lu and Gu} \\ \langle Gu,v\rangle_\theta & = \int_M c\vert u\vert^{p-2}uv\; dV_g, 
\end{split}
\end{align}
\end{linenomath}
for every $v\in H_{g}^{1}(M)$. It follows that 
\[
\nabla J(u)=u-Lu-Gu, \quad\text{ for every }u\in H^1_g(M).
\]
If $u\in H^1_g(M)^\fol$, then also $(\theta-b)u$ and  $c|u|^{p-2}$ are $\fol$-invariant. Thus, Lemma \ref{L: Aux PDE's} yields that the functions $Lu$ and $Gu$ are $\fol$-invariant, and so we conclude that for $u\in H_g^1(M)^\fol$, the function $\nabla J(u)$ is $\fol$-invariant  as we claimed.

Recall that the restricted Nehari manifold was defined to be 
\[
\mathcal{N}^\fol_g=\{u\in H_g^1(M)^\fol \mid u\neq0,\ \Vert u\Vert_b^2 = \vert u\vert_{c,p}^p \}.
\]
Observe that if $u\in H_g^1(M)^\fol$, then also $u^{\pm}\in H_g^1(M)^\fol$, where $u^{+}:=\max\{0,u\},$ $u^{-}:=\min\{0,u\}$. The nontrivial $\fol$-invariant sign-changing critical points of $J$, and thus, sign changing solutions to \eqref{Main:Yamabe type} in virtue of Theorem \ref{Theorem:Symmetric Criticality}, must belong to the set 
\begin{linenomath}
\[
    \mathcal{E}_g^{\fol}:=\{u\in\mathcal{N}_g^{\fol}\mid u^{+},u^{-}%
\in\mathcal{N}_g^{\fol}\}.
\]
\end{linenomath}
This set is nonempty. Indeed, Theorem \ref{Lemma:Infinite dimensional} gives the existence of at least two foliated smooth functions $u_1, u_2\geq 0$ with disjoint supports. Then recalling the definition of the projection onto the Nehari manifold given in \eqref{eq: Projection Nehari}, we can define the function $u:=\sigma(u_1)-\sigma(u_2)$, which is an element in $\mathcal{E}_g^\fol$, given that $u^+=\sigma(u_1)\in\mathcal{N}_g^\fol$ and $u^- = -\sigma(u_2)\in \mathcal{N}_g^{\fol}$.

Let $\mathcal{P}^{\fol}:=\{u\in H^1_g(M)^\fol\mid u\geq0\}$ be the convex cone of nonnegative
functions. Then the set of functions in $H^1_g(M)^\fol$ which do not change sign is given as
$\mathcal{P}^{\fol}\cup-\mathcal{P}^{\fol}$.

As the $\langle\cdot,\cdot\rangle_\theta$-gradient of $J$, $-\nabla J\colon H_g^1(M)^\fol\to H_g^1(M)^\fol$,  is of class $C^1$, it is locally Lipschitz and thus for each $u\in H_g^1(M)^\fol$, the Cauchy problem:  
\[
\begin{cases}
\frac{\partial}{\partial t}\psi(t,u) = -\nabla J(\psi(t,u))\\
\psi(0,u)=u,
\end{cases}
\] 
has a unique solution defined for all $0\leq t < T(u)$, where $T(u)\in(0,\infty)$ is the maximal existence time of the solution.  
With this at hand, the \emph{negative gradient flow of $J$}, is simply the map $\psi\colon \mathcal{G}\to H^1_g(M)^\fol$ where 
$\mathcal{G}:=\{(t,u) \mid u\in H^1_g(M)^\fol,$ $0\leq t<T(u)\}$ and $\psi$ is as above. 
A subset $\mathcal{D}$ of $H^1_g(M)^\fol$ is said to be \emph{strictly
positively invariant under $\psi$} if
\[
\psi(t,u)\in\text{int}\mathcal{D}\text{\qquad for every }u\in\mathcal{D}\text{
and }t\in(0,T(u)).
\]

For any set $\mathcal{D} \subset H_g^1(M)^\fol$, we define the set
\[
\mathcal{A}(\mathcal{D}):=\{u\in H_g^1(M)^\fol \mid \psi(t,u)\in\mathcal{D} \text{ for some }t\in(0,T(u))\},
\]
and the entrance time $e_{\mathcal{D}}\colon\mathcal{A}(\mathcal{D})\to\mathbb{R}$, given by
\[
e_{\mathcal{D}}(u):=\inf\{t\geq 0\mid \psi(t,u)\in\mathcal{D}\}.
\]
If $\mathcal{D}$ is strictly positively invariant under $\psi$,
then $\mathcal{A}(\mathcal{D})$ is open in $H_g^1(M)^\fol$ and $e_{\mathcal{D}}(u)$ is continuous.

In what follows, for any subset $\mathcal{D}$ of $H^1_g(M)^\fol$ and any $\varepsilon>0$, $\overline B_{\varepsilon}(\mathcal{D})$ will denote the set
\[
\overline B_{\varepsilon}(\mathcal{D}):=\{u\in H^1_g(M)^\fol\mid\text{dist}_{\theta}%
(u,\mathcal{D}):=\inf_{v\in\mathcal{D}}\left\Vert
u-v\right\Vert_\theta\leq\varepsilon\}.
\]

To find sign-changing critical points of $J$ we use the relative genus between symmetric subsets of $H_{g}^{1}(M)^{\fol}$. But before that we need the following preliminary result that we prove in Subsection \ref{sub-appendix}. 

\begin{lemma}\th\label{lem-Invariance}
\th\label{lem:variational_principle}
There exists $\alpha_{0}>0$ such that for every $\alpha\in(0,\alpha_{0}),$
\begin{enumerate}
\item[(a)] $\overline B_{\alpha}(\mathcal{P}^{\fol}) \cap\mathcal{E}_g^{\fol}=\emptyset$,  
 $\overline B_{\alpha}(-\mathcal{P}%
^{\fol})\cap\mathcal{E}_g^{\fol}=\emptyset$, and
\item[(b)] $\overline B_{\alpha}(\mathcal{P}^{\fol})$ and $\overline B_{\alpha}(-\mathcal{P}%
^{\fol})$ are strictly positively invariant  under the flow of the negative gradient of $J$ with respect to $\langle\cdot,\cdot\rangle_\theta$.
\end{enumerate}
\end{lemma}

Fix $\alpha$ as in the previous lemma. For $d\in\mathbb{R}$, set
\[
\mathcal{D}_{d}^{\fol}:=\overline B_{\alpha}(\mathcal{P}^{\fol})\cup \overline B_{\alpha
}(\mathcal{-P}^{\fol})\cup J^{d},
\]
where $J^{d}:=\{u\in H^1_g(M)^\fol\mid J(u)\leq d\}$. 

The next result says that, under suitable conditions, $\mathcal{D}_d^\fol$ is a neighborhood retract.

\begin{lemma}\th\label{Lemma:D positively invariant}
If $J$ has no sign-changing critical point $u\in H_g^1(M)^\fol$ with $J(u)=d$, then $\mathcal{D}_d^\fol$ is strictly positively invariant under $\psi$ and the map 
\[
\rho_d\colon\mathcal{A}(\mathcal{D}_d^\fol)\to\mathcal{D}_d^\fol,\quad \rho_d(u):=\psi(e_{\mathcal{D}_d^\fol}(u),u),
\]
is odd and continuous, and satisfies $\rho_d(u)=u$ for every $u\in\mathcal{D}_d^\fol$.
\end{lemma}

\begin{proof}
Note that by definition $\rho_d$ is odd. 
To show that $\mathcal{D}_d^\fol$ is strictly positively invariant under $\psi$, by  \th\ref{lem-Invariance}, it suffices to consider $u\in J^d$.  By definition of the flow, given any $u\in J^d$, we have that
\begin{equation}\label{Eq:Decreasing Flow}
\begin{split}
\frac{d}{dt}J\circ\psi(t,u) 
&= J'(\psi(t,u))\frac{\partial}{\partial t}\psi(t,u)\\
&= J'(\psi(t,u))(-\nabla J(\psi(t,u)))\\
& = -\Vert \nabla J(\psi(t,u))\Vert_{\theta}^2\leq 0,\quad t\in[0,T(u)).
\end{split}
\end{equation}
Thus we conclude that $\psi(t,u)\in J^d$ for every $t\in[0,T(u))$. So, if $J(u)<d$, it follows that $\psi(t,u)\in \text{int}J^d$ for every $t\in[0,T(u))$. Next, suppose that $J(u)=d$. By hypothesis, $u$ cannot be a sign changing critical point for $J$ in $H_g^1(M)^\fol$, hence we have two possible cases: either $u$ is a critical point that does not change sign, or $J'(u)\neq 0$. 

The first case reduces to the first paragraph of the proof, that is we have $u\in\overline{B}_\alpha(\mathcal{P}^\fol)\cup\overline{B}_\alpha(-\mathcal{P}^\fol)$, and hence there is nothing to prove.  
Now, suppose that the second case holds true. Then, there exist real numbers $\beta,\varepsilon>0$ such that $\Vert \nabla J(v)\Vert_\theta>\beta$ for every $v\in [d-\varepsilon,d+\varepsilon]$. Thus, $\Vert \nabla J(\psi(t,u))\Vert_\theta\geq\beta>0$ for $t$ small enough, and \eqref{Eq:Decreasing Flow} yields that $\psi(t,u)\in\text{int}J^d$. Thus, $\mathcal{D}_d^\fol$ is strictly positively invariant under $\psi$.

It is now easy to check that $\rho_d$ has the desired properties.
\end{proof}

\begin{remark}\th\label{Rem:D positively invariant}
Observe that $\mathcal{D}_0^\fol$ is always strictly positively invariant under the flow $\psi$, for there are no nontrivial critical points of $J$ satisfying $J(u)=0$. In fact, if $u$ is a critical point of $J$ satisfying  $J(u)=0$, then $0=J'(u)u=\Vert u\Vert_b -\vert u\vert_{c,p}$; this implies that $J(u)=\frac{p-2}{2p}\Vert u\Vert_b=0$ and $u=0$. 
\end{remark}

\begin{remark}\th\label{Rem:SignChanging complement D}
Notice that a critical point of $J$ changes sign if and only if it lies in the complement of $\mathcal{D}_{0}^{\fol}$ in $H^1_g(M)^\fol$. Indeed, if $u$ lies in $H^1_g(M)^\fol\smallsetminus\mathcal{D}_0^\fol$, then $u\notin\mathcal{P}^\fol\cup(-\mathcal{P}^\fol)$ and it must change sign. Conversely, if $u$ is a sign changing critical point of $J$, as we pointed before, it belongs to $\mathcal{E}_g^\fol$ and \th\ref{lem-Invariance} says that this set has empty intersection with $\overline{B}_\alpha(\mathcal{P}^\fol)\cup\overline{B}_\alpha(-\mathcal{P}^\fol)$; moreover, as it is a nontrivial critical point, we have that $J(u)=\frac{p-2}{2p}\Vert u\Vert_b>0$ and $u\notin J^0$. Thus, $u\notin \mathcal{D}_0^\fol$.
\end{remark} 

A subset $\mathcal{Y}$ of $H_{g}%
^{1}(M)^{\fol}$ will be called symmetric if $-u\in\mathcal{Y}$ for every $u\in\mathcal{Y}$. 

\begin{definition}
Let $\mathcal{D}$ and $\mathcal{Y}$ be symmetric subsets of $H_{g}%
^{1}(M)^{\fol}$. The genus of $\mathcal{Y}$ relative to $\mathcal{D}$,
denoted by $\mathfrak{g}(\mathcal{Y},\mathcal{D})$, is the smallest number $n$
such that $\mathcal{Y}$ can be covered by $n+1$ open symmetric subsets
$\;\mathcal{U}_{0},\mathcal{U}_{1},\ldots,\mathcal{U}_{n}$ of $H_{g}%
^{1}(M)^{\fol}$ with the following two properties:

\begin{itemize}
\item[(i)] $\mathcal{Y}\cap\mathcal{D}\subset\mathcal{U}_{0}$ and there exists
an odd continuous map $\vartheta_{0}\colon\mathcal{U}_{0}\to\mathcal{D}$
such that $\vartheta_{0}(u)=u$ for $u\in\mathcal{Y}\cap\mathcal{D}$.

\item[(ii)] There exist odd continuous maps $\vartheta_{j}\colon \mathcal{U}%
_{j}\to \{1,-1\}$ for every $j=1,\ldots,n$.
\end{itemize}
If no such cover exists, we define $\mathfrak{g}(\mathcal{Y},\mathcal{D}%
):=\infty$.
\end{definition}

As in \cite[Section 3]{ClappPacella2008} in order to obtain a variational principle for sign changing solutions, we need a refined version of the $(PS)^\fol_\tau$ condition. Given $\mathcal{D}\subseteq H_g^1(M)^\fol$, we say that $J$ satisfies the \emph{$(PS)_\tau^\fol$ condition relative to $\mathcal{D}$ in $H_g^1(M)^\fol$}, if every sequence $(u_n)_{n\in\N}$ in $H_g^1(M)^\fol$ such that
\[
u_n\notin\mathcal{D},\quad J(u_n)\to \tau,\quad J'(u_n)\to 0 \text{ in }(H_g^1(M)^\fol)^\ast
\]
has a strongly convergent subsequence  in $H^1_g(M)$. When $\mathcal{D}=\emptyset$, we recover the $(PS)_\tau^\fol$ condition given in Section \ref{Sec:Main Results}. Also notice that if $J$ satisfies the $(PS)_\tau^\fol$ condition in $H_g^1(M)^\fol$, then it satisfies the condition relative to any subset of $H_g^1(M)^\fol.$

\begin{lemma}\th\label{Lemma:Var pri crit points}
Fix $\alpha\in(0,\alpha_{0})$ with $\alpha_0$ given as in Lemma \ref{lem-Invariance}. For $j \in \mathbb N$, define
\[
\tau_{j}:=\inf\{\tau\in\mathbb{R}\mid \mathfrak{g}(\mathcal{D}_{\tau}^{\fol}%
,\mathcal{D}_{0}^{\fol})\geq j\}.
\]
Assume that $J$ satisfies $(PS)_{\tau_{j}}^{\fol}$ relative to $\mathcal{D}_0^\fol$ in $H^1_g(M)^\fol$. Then, the following statements hold true:
\begin{itemize}
\item[(a)] $J$ has a sign-changing critical point $u\in H_{g}%
^{1}(M)^{\fol}$ with $J(u)=\tau_{j}$.
\item[(b)] If $\tau_{j}=\tau_{j+1}$, then $J$ has infinitely many sign-changing
critical points $u\in H^1_g(M)^\fol$ with $J(u)=\tau_{j}$.
\end{itemize}
Consequently, for $d\in\mathbb{R}$, if $J$ satisfies $(PS)_{\tau}^{\fol}$ relative to $\mathcal{D}_0^\fol$ in $H^1_g(M)^\fol$ for
every $\tau\leq d$, then $J$ has at least $\mathfrak{g}(\mathcal{D}%
_{d}^{\fol},\mathcal{D}_{0}^{\fol})$ distinct pairs of sign-changing critical
points: $\pm u_1,...,\pm u_k$, $k \geq \mathfrak{g}(\mathcal{D}%
_{d}^{\fol},\mathcal{D}_{0}^{\fol})$, in $H^1_g(M)^\fol$ with $J(\pm u_i)\leq d$ for all $i=1,...,k$. 
\end{lemma}

\begin{proof}
The proof is exactly the same as that of Proposition~3.6 in \cite{ClappPacella2008}, but we include it for the sake of completeness. 

We prove part $(a)$ by contradiction,  using the fact that $\mathcal{D}_{0}^{\fol}$ is strictly positively invariant under the flow $\psi$ (see \th\ref{Rem:D positively invariant}). More precisely, we assume that there does not exist a sign-changing critical point $u\in H^1_g(M)^\fol$ with $J(u)=\tau_j$. We claim that the $(PS)_{\tau_j}^\fol$ condition relative to $\mathcal{D}_0^\fol$ implies the existence of $\beta>0$ and $\varepsilon\in (0,\tau_j)$ such that 
\begin{equation}\label{Eq:Nonstopping flow}
\Vert \nabla J(u)\Vert_{\theta}\geq \beta>0 ,\quad\text{for every } u\in J^{-1}[ \tau_{j}-\varepsilon ,\tau_{j}+\varepsilon] \setminus\mathcal{D}_0^\fol.
\end{equation}
To see this, suppose in order to get a contradiction, the existence of a sequence $(u_n)_{n\in\N}$ in $H_g^1(M)^\fol$ such that $J(u_n)\to \tau_j$ and $\Vert \nabla J(u_n)\Vert_\theta\to 0$. Then $(u_n)_{n\in\N}$ is a $(PS)_{\tau_j}$ sequence, and since $J$ satisfies the $(PS)_{\tau_j}$ condition relative to $\mathcal{D}_0^\fol$ there exists $u\in H_g^1(M)^\fol$ such that, up to a subsequence which we also denote by $u_n$, it holds that $u_n\to u$ strongly in $H_g^1(M)$. As $J$ and $J'$ are continuous, it follows that $J(u)=\tau_j$ and $J'(u)=0$. By hypothesis, $u$ cannot change sign. Hence $u$ is a critical point of $J$ lying in $\mathcal{P}^\fol\cup(-\mathcal{P}^\fol )\subset \text{int}\mathcal{D}_0^\fol$, which is open. Therefore, there exists $n_0$ such that $u_n\in \mathcal{D}_0^\fol$ for every $n\geq n_0$, which is a contradiction and the claim follows. 

As we pointed out in  \th\ref{Rem:SignChanging complement D}, every sign changing critical point lies in $H_g^1(M)^\fol\smallsetminus \mathcal{D}_{0}^\fol$, and \eqref{Eq:Nonstopping flow} implies that there are no sign changing critical points of $J$ in $J^{-1}[\tau_j-\varepsilon,\tau_j+\varepsilon]$. Therefore $\mathcal{D}_d^\fol$ is strictly positively invariant under $\psi$ for every $d\in[\tau_j-\varepsilon,\tau_j+\varepsilon]$ by \th\ref{Lemma:D positively invariant}. As $\Vert \nabla J\Vert_\theta>0$, identity \eqref{Eq:Nonstopping flow} yields that $\mathcal{D}_{\tau_j+\varepsilon}^\fol$ flows under $\psi$ to $\mathcal{D}_{\tau_j-\varepsilon}^\fol$ and $\mathcal{D}_{\tau_j+\varepsilon}^\fol\subset\mathcal{A}(\mathcal{D}_{\tau_j-\varepsilon}^\fol)$. In this way, \th\ref{Lemma:D positively invariant} implies that $\rho_{\tau_j-\varepsilon}\colon \mathcal{D}_{\tau_i+\varepsilon}^\fol\to\mathcal{D}_{\tau_i-\varepsilon}^\fol$ is odd, continuous and $\rho_{\tau_j-\varepsilon}(u)=u$ for every $u\in\mathcal{D}_0^\fol$, given that $\mathcal{D}_0^\fol\subset\mathcal{D}_{\tau_i-\varepsilon}^\fol$. As $\mathcal{D}_0^\fol$ is a symmetric neighborhood retract by \th\ref{Rem:D positively invariant} and \th\ref{Lemma:D positively invariant}, from the monotonicity of the genus $\mathfrak{g}$ \cite[Lemma 3.4]{ClappPacella2008} it follows that $j \leq \mathfrak{g}(\mathcal{D}_{\tau_j+ \varepsilon}^{\fol},\mathcal{D}_{0}^{\fol}) \leq \mathfrak{g}(\mathcal{D}_{\tau_j - \varepsilon}^{\fol},\mathcal{D}_{0}^{\fol})<j$ which is a contradiction. 

The proof of part $(b)$ is exactly the same as the proof of Proposition 3.6 (b) in \cite{ClappPacella2008}. The argument uses again the fact that $\mathcal{D}_{0}^{\fol}$ is strictly positively invariant under the flow $\psi$, the contradicting hypothesis, and, the monotonicity and subadditive properties of the genus  (see \cite[Lemma 3.4]{ClappPacella2008}) to get $j+1 \leq \mathfrak{g}(\mathcal{D}_{\tau_j+ \varepsilon}^{\fol},\mathcal{D}_{0}^{\fol}) \leq \mathfrak{g}(\mathcal{D}_{\tau_j - \varepsilon}^{\fol},\mathcal{D}_{0}^{\fol}) +1 <j+1$. 
\end{proof}

We now prove \th\ref{Theorem:Variational Principle}. We start by showing the existence of a positive solution of minimal energy using the ideas developed in Section~\ref{Sec:Compactness}. Then we show the existence of the sign-changing solutions following the proof of Theorem 3.7 in \cite{ClappPacella2008}.

\begin{proof}
[Proof of Theorem \ref{Theorem:Variational Principle}]
We first show the existence of a positive critical point attaining $\tau_g^\fol = \inf\{J(u)\mid u\in \mathcal{N}_g^\fol\}$. Let $u_n\in\mathcal{N}_g^\fol$ such that $J(u_n)\rightarrow\tau_g^\fol$. Since $u_n\in\mathcal{N}_g^\fol$ 
we get
\[
J(u_n)=\frac{1}{2} \Vert u_n\Vert_b^2 - \frac{1}{p}\vert u_n\vert_{c,p}^p
=\frac{p-2}{2p}\vert u_n\vert_{c,p}^p.
\]
Thus, $\vert u_n\vert_{c,p}^p$ is bounded in $H_g^1(M)^\fol$. 
As $2<p\leq 2_m^\ast$ and $\kappa_\fol\geq 1$, as in the proof of Theorem \ref{Theorem:Compactness}, there exists $u\in H_g^1(M)^\fol$ such that $u_n$ converges weakly to $u$ in $H_g^1(M)^\fol$ and strongly in $L_g^p(M)$. Then 
\[
\vert u\vert_{c,p}^p =\lim_{n\rightarrow\infty}\vert u_n\vert_{c,p}^p=\frac{2p}{p-2}\tau_g^\fol>0,
\]
implying that $u\neq 0$. Therefore, there exists $t_u>0$ such that $\sigma(u)=t_u u\in\mathcal{N}_g^\fol$, where $\sigma$ is the projection onto $\mathcal{N}_g$.
 As $u_n\in \mathcal{N}_g^\fol$, identity \eqref{Eq:Nehari maximum} yields that $J(t_u u_n)\leq J(u_n)$. Therefore, using basic properties of weak convergence and that $u_n\rightarrow u$ in $L_g^p(M)$, we obtain that
\[
\begin{split}
\tau_g^\fol&\leq J(t_u u) = \frac{1}{2}\Vert t_u u\Vert_b^2 - \frac{1}{p}\vert t_u u\vert_{c,p}^p\\
&\leq \liminf_{n\rightarrow\infty}\left( \frac{1}{2}\Vert t_u u_n\Vert_b^2 \right) - \liminf_{n\rightarrow\infty}\left( \frac{1}{p}\vert t_u u_n\vert_{c,p}^p \right)\\
&=\liminf_{n\rightarrow\infty}J(t_uu_n)\leq \lim_{n\rightarrow\infty}J(u_n)=\tau_g^\fol.
\end{split}
\]
Again, since $u_n\rightarrow u$ strongly in $L^p_g(M)$, we conclude from this inequalities that $\lim_{n\rightarrow\infty}\Vert u_n\Vert^2_b$ exists and is equal to $\Vert u\Vert^2_b$. Hence $u_n\rightarrow u$ strongly in $H_g^1(M)^\fol$, and since $\mathcal{N}_g^\fol$ is closed in $H^1_g(M)^\fol$ then it follows that $u\in\mathcal{N}_g^\fol$, and $J(u)=\tau_g^\fol$. As $\mathcal{N}_g^\fol$ is a natural restriction for the functional $J$ (see \cite[Chapter 4]{Willem1996}), $u$ is a nontrivial $\fol$-invariant critical point for $J$ in $H_g^1(M)^\fol$ attaining $\tau_g^\fol$.
\smallskip

Now we see that $u$ does not change sign. Suppose, in order to get a contradiction, that this is not true. Then, $u^+\neq0$ and $u^-\neq 0$. As $u^{\pm}\in H_g^1(M)$ and as $u$ is a critical point for $J$, we have that
\[\begin{split}
0&=J'(u)u^{\pm}=\int_M [ \langle \nabla_g u, \nabla_g(u^\pm) \rangle_g + b u (u^\pm) ] - \int_M c \vert u\vert^{p-2} u (u^\pm) \\
&= \Vert u^\pm\Vert_b^2 - \vert u^{\pm}\vert_{c,p}^p,
\end{split}\]
concluding that $u^\pm\in\mathcal{N}_g^\fol$. Hence 
\[
\tau_g^\fol = J(u)= J(u^+) + J(u^-)\geq 2 \inf_{v\in\mathcal{N}_g^\fol}J(v)=2\tau_g^\fol,
\]
which is a contradiction since $\tau_g^\fol>0$. Thus $u$ does not change sign. If $u\leq 0$, we can take $-u$, since it is also a critical point for $J$ and it is positive.

\medskip
We proceed to prove the existence of the sign-changing critical points:
\medskip
Let $d:=\sup_{W}J$. By Lemma
\ref{Lemma:Var pri crit points}, we only need to show that 
$$n:=\mathfrak{g}
\left(  \mathcal{D}_{d}^{\fol},\mathcal{D}_{0}^{\fol}\right)  \geq
\dim(W)-1.$$
Let $\mathcal{U}_{0},\mathcal{U}_{1},\ldots,\mathcal{U}_{n}$ be
open symmetric subsets of $H^1_g(M)^\fol$ covering $\mathcal{D}%
_{d}^{\fol}$ with $\mathcal{D}_{0}^{\fol}\subset\mathcal{U}_{0}$ and let
$\vartheta_{0}\colon \mathcal{U}_{0}\to \mathcal{D}_{0}^{\fol}$ and
$\vartheta_{j}\colon \mathcal{U}_{j}\to \{1,-1\}$, $j=1,\ldots,n$, be odd
continuous maps such that $\vartheta_{0}(u)=u$ for all $u\in\mathcal{D}%
_{0}^{\fol}$. Since $H^1_g(M)^\fol$ is an absolute retract, we may assume that
$\vartheta_{0}$ is the restriction of an odd continuous map $\widetilde
{\vartheta}_{0}\colon H^1_g(M)^\fol\to H^1_g(M)^\fol$. Let
$\mathcal{B}$ be the connected component of the complement of the Nehari
manifold $\mathcal{N}_{g}^{\fol}$ in $H^1_g(M)^\fol$ which contains
the origin, and set $\mathcal{O}:=\{u\in W\mid\widetilde{\vartheta}_{0}%
(u)\in\mathcal{B}\}$. Then, $\mathcal{O}$ is a bounded open symmetric
neighborhood of $0$ in $W$.

Let $\mathcal{V}_{j}:=\mathcal{U}_{j}\cap\partial\mathcal{O}$. Then,
$\mathcal{V}_{0},\mathcal{V}_{1},\ldots,\mathcal{V}_{n}$ are symmetric and
open in $\partial\mathcal{O}$, and they cover $\partial\mathcal{O}$. Further,
by Lemma \ref{lem:variational_principle},
\[
\vartheta_{0}(\mathcal{V}_{0})\subset\mathcal{D}_{0}^{\fol}\cap
\mathcal{N}_{g}^{\fol}\subset\mathcal{N}_{g}^{\fol}\smallsetminus
\mathcal{E}_{g}^{\fol}.
\]
The set $\mathcal{N}_{g}^{\fol}\smallsetminus\mathcal{E}_{g}^{\fol}$
consists of two connected components, see for example \cite[Lemmas 2.5 and 2.6]{CastroCossioNeuberger1997}. Therefore there
exists an odd continuous map $\eta\colon\mathcal{N}_{g}^{\fol}\smallsetminus
\mathcal{E}_{g}^{\fol}\to\{1,-1\}$. Let $\eta_{j}\colon\mathcal{V}%
_{j}\to\{1,-1\}$ be the restriction of the map $\eta\circ\vartheta
_{0}$ if $j=0,$ and the restriction of $\vartheta_{j}$ if $j=1,\ldots,n$. Take
a partition of the unity $\{\pi_{j}\colon\partial\mathcal{O}\to
\lbrack0,1]\mid j=0,1,\ldots,n\}$ subordinated to the cover $\{\mathcal{V}%
_{0},\mathcal{V}_{1},\ldots,\mathcal{V}_{n}\}$ consisting of even functions,
and let $\{e_{1},\ldots,e_{n+1}\}$ be the canonical basis of $\mathbb{R}%
^{n+1}$. Then, the map $\Psi\colon\partial\mathcal{O}\to\mathbb{R}^{n+1}$
given by
\[
\Psi(u):=\sum_{j=0}^{n}\eta_{j}(u)\pi_{j}(u)e_{j+1}%
\]
is odd and continuous, and satisfies $\Psi(u)\neq0$ for every $u\in
\partial\mathcal{O}$. The Borsuk-Ulam theorem allow us to conclude that
$\dim(W)\leq n+1,$ as claimed.
\end{proof}

\subsection{Proof of auxiliary results}\label{sub-appendix}

\begin{proof}[Proof of Lemma \ref{L: Aux PDE's}]
Fix $f\in L^2_g(M)$, and define $I_f\colon H^1_g(M)\to \R$ as 
\begin{linenomath}
\[
I_f(u) = \int_M fu dV_g.  
\]
\end{linenomath}
Observe that $I_f$ is linear and bounded because $f\in L^2(M)$ and since for any $u\in H^1_g(M)$ it holds that $u\in L_g^2(M)$.  By the
Frèchet-Riesz representation theorem, there exists a unique $v_f\in H^1_g(M)$ such that the following equality holds for any $u\in H^1_g(M)$,
\begin{linenomath}
\[
    \langle v_f, u\rangle_\theta  = I_f(u). 
\]
\end{linenomath}
This implies that $v_f$ is a weak solution in $M$ to the equation
\begin{linenomath}
\begin{equation*}
    -\Delta_g v + \theta v = f.
\end{equation*}
\end{linenomath}

Now we show that if $f$ is $\fol$-invariant then $v_f$ is $\fol$-invariant. Consider the operator $J_\theta\colon H^1_g(M)\to \R$ given by \begin{linenomath}
\[
J_\theta(u) = \tfrac{1}{2}\Vert u \Vert_\theta - I_f(u).
\]
\end{linenomath}
Observe that for any $u,w\in H^1_g(M)$ we have that
\begin{linenomath}
\[
J'_\theta(w)(u) = \langle w,u\rangle_\theta - I'_f(w)(u).
\]
\end{linenomath}
Note that for any $u,w\in H^1_g(M)$,
\[
    I'_f(w)(u) = I_f(u).
\]
Thus we see that $v_f$ is a critical value of $J_\theta$ by construction.

Recall that $H = H^1_g(M)^\fol$ is a closed linear subspace
of $H^1_g(M)$.  We note that the orthogonal decomposition $H^1_g (M)= H\oplus H^\perp$ with respect to the standard inner product is also an orthogonal decomposition with respect to $\langle\,,\,\rangle_\theta$. Indeed if $u\in H$ and $v\in H^\perp$, so that $\langle u,v\rangle_{H^1_g(M)} = 0$,
then we have 
\[ 
\langle u,v\rangle_{\theta} =  \int_M (\theta -1) uv dV_g.
\]
Since $\theta>1$ we conclude that the previous expression equals zero applying Lemma \ref{lem:H1Ortho}. 

Write $v_f = v_f^\top+v_f^\perp$, with $v_f^\top \in H$ and $v_f^\perp \in H^\perp$.
We claim that $v^\top_f$ is a critical point of $J_\theta$. If the claim holds, then for any $u\in H^1_g(M)$ we have that $I_f(u) = \langle v^\top_f,u\rangle_\theta$, and since we know that $I_f(u)=\langle v_f, u\rangle_\theta$ for any $u\in H^1_g(M)$, we conclude that $v_f=v_f^\top$, that is, $v_f\in H^1_g(M)^\fol$ as desired.

Now we prove the claim. For $u\in H$,
\begin{linenomath}
\begin{align*}
J'_\theta(v^\top_f)(u) = & \langle v_f^\top, u \rangle_\theta- I'_f(v^\top_f)(u) \\
= & \langle v_f, u\rangle_\theta- I'_f(v_f)(u) = J'(v_f)(u)=0.
\end{align*}
\end{linenomath}
This implies that $v^\top_f$ is a critical point of $J_\theta|_{H^{1}_{g}(M)^{\fol}}$. 
 Now consider a fixed $w\in H^\perp$. Then, 
 $$J'_\theta(v^\top_f)(w) = \langle v^\top_f, w\rangle_\theta - I'_f(v^\top_f)(w)=I_f(w)=\int_M f w dV_g.$$
  Since $f\in L^2(M)^\fol$,  \th\ref{L: H1(M) fol perp in L2(M) fol perp}
  implies that $v^\top_f$ is a critical point of $J_\theta|_{H^\perp}$.
 By linearity, it follows that $v^\top_f$  is a critical point of $J_\theta$.
 \end{proof}

We proceed to prove Lemma \ref{lem-Invariance}. The proof, up to minor modifications, is the same as in \cite[Lemma 5.2]{ClappFernandez2017}, but we sketch it for the convenience of the reader.

\begin{proof}[Proof of Lemma \ref{lem-Invariance} part $(a)$]

For every $u\in H^1_g(M)^\fol$, Sobolev's inequality yields a positive constant $C$ such that
\begin{equation}\label{eq:distP} 
\text{dist}_{\theta}(u,\mathcal{P}^{\fol
}) = \min_{v\in\mathcal{P}^{\fol
}}\left\Vert u-v\right\Vert_\theta \geq 
 C^{-1}\min_{v\in\mathcal{P}^{\fol}}\left\vert u-v\right\vert _{c,p}  \geq C^{-1} \vert u^{-}\vert _{c,p},
\end{equation}
where in the last part we used the fact that $\left\vert
u-v\right\vert \geq \left\vert u^{-}\right\vert$ for every $u,v\colon M\to\mathbb{R}$ with $v\geq0$, and $u^-=-\min\{0,u\}$.

Now we assume that $u\in\mathcal{E}_g^{\fol}$. 
We bound $\vert u^{-}\vert _{c,p}$
by noticing that $u^{-}\in\mathcal{N}_g^{\fol}$
implies that, 
\[
J(u^{-})= 
\tfrac{1}{2}\Vert u^{-}\Vert^2_b -\tfrac{1}{p}\vert u^{-}\vert_{c,p}^p 
= \tfrac{p-2}{2p}\vert u^{-}\vert^p_{c,p}.
\]
 Then by definition of $\tau_g^\fol$, see \eqref{def:LeastEnergy}, we get
\[
\vert
u^{-}\vert _{c,p}^p = \tfrac{2p}{p-2}J(u^{-})\geq \tfrac{2p}{p-2}\tau_g^\fol>0.
 \]
Taking $\alpha_1:=C^{-1}\left(\tfrac{2p}{p-2}\tau_g^\fol\right)^{\tfrac{1}{p}}>0$, it follows that $\mathrm{dist}_{\theta}(u,\mathcal{P}^{\fol})\geq
\alpha_{1}$ for all $u\in\mathcal{E}^{\fol}_g$.
Similarly, $\mathrm{dist}_{\theta}(u,-\mathcal{P}^{\fol})\geq
\alpha_{1}$ for all $u\in\mathcal{E}^{\fol}_g$.
Thus, 
$\overline B_{\alpha}(\mathcal{P}^{\fol}) \cap\mathcal{E}_g^{\fol}, \,\overline B_{\alpha}(-\mathcal{P}%
^{\fol})\cap\mathcal{E}_g^{\fol}=\emptyset$, 
for every $\alpha\in(0,\alpha_{1})$.
\end{proof}

To prove part $(b)$ of Lemma \ref{lem-Invariance} we will use the following result.

 \begin{proposition}[Theorem 5.2 in \cite{Deimling}]\th\label{T: Deimling}
 Let $X$ be a real vector space with a norm inducing a distance $\rho$, $\Omega \subset X$ open and $D \subset X$ closed convex with non empty interior
 with $D \cap \Omega \neq \emptyset$ and such that the distance from any point in $X$ to $D$ is achieved by some point. Let $f \colon (0,a) \times \Omega \to X$ be 
 a locally Lipschitz function that satisfies
 \[\lim_{\lambda \to 0}\lambda^{-1}\rho(u+\lambda f(t,u), D)=0  \quad \text{for} \quad u \in \Omega \cap \partial D \quad \text{and} \quad t \in (0,a).
 \]
 Then any continuous $x\colon [0,b) \to \Omega$ such that
 $x(0) \in D$ and $x'(t)=f(t,x(t))$ in $(0,b)$ satisfies $x(t) \in D$
 for all $t\in [0,b)$. 
 \end{proposition}

\begin{proof}[Proof of Lemma \ref{lem-Invariance} part $(b)$]
By symmetry, we will only prove this part for
$\mathcal{P}^{\fol})$. The proof consists in the following two steps:

{\bf{Step 1:}} For $\alpha>0$ small enough, we apply \th\ref{T: Deimling}
    taking $X=\Omega=H_g^1(M)^\fol$, $\rho=\text{dist}_\theta$, and $D=\overline B_{\alpha}(\mathcal{P}^{\fol})$
    which is closed and
convex, and $f(t,u)=-\nabla J(u)$, to obtain that 
    \[
\psi(t,u)\in \overline B_{\alpha}(\mathcal{P}^{\fol})\text{ \ for all \ }t\in
(0,T(u))\text{ \ if }u\in \overline B_{\alpha}(\mathcal{P}^{\fol}).
\]

{\bf{Step 2:}} We then apply Mazur's separation theorem \cite[Section 2.2.19]{Megginson} to prove that 
\[
\psi(t,u)\in B_{\alpha}(\mathcal{P}^{\fol})\text{ \ for all \ }t\in
(0,T(u))\text{ \ if }u\in \partial \overline B_{\alpha}(\mathcal{P}^{\fol}).
\]

\medskip

{\bf{Proof of Step 1:}} Recall that the gradient of $J$ with respect to $\langle\cdot,\cdot\rangle_\theta$ is $\nabla J (u) = u-Lu-Gu$.  In order to apply 
 Proposition \ref{T: Deimling}, for $u\in\partial \overline{B}_\alpha(\mathcal{P}^\fol)$, we have to show that the 
  limit as $\lambda \searrow 0$ of the following expression equals zero:
\[
\begin{split}
&\lambda^{-1}\text{dist}_\theta\left(u+\lambda(-\nabla J(u)),\overline{B}_\alpha(\mathcal{P}^\fol)\right) \\
&= \lambda^{-1}\text{dist}_\theta\left((1-\lambda) u + \lambda(Lu+Gu),\overline{B}_\alpha(\mathcal{P}^\fol)\right)\\
&\leq \lambda^{-1}\left[ (1-\lambda)\text{dist}_\theta\left( u ,\overline{B}_\alpha(\mathcal{P}^\fol)\right) + \lambda\text{dist}_\theta \left(Lu+Gu,\overline{B}_\alpha(\mathcal{P}^\fol)\right) \right]\\
& = \text{dist}_\theta \left(Lu+Gu,\overline{B}_\alpha(\mathcal{P}^\fol)\right).
\end{split}
\]
This will be achieved by showing that  $Lu+Gu\in\,B_{\alpha}(\mathcal{P}^{\fol})$. In particular, we show that for some $\alpha_2 >0$ and any
$\alpha\in(0,\alpha_{2})$,  
\[
\text{dist}_{\theta}(Lu+Gu,\mathcal{P}^{\fol})< \text{dist}_{\theta}%
(u,\mathcal{P}^{\fol})\text{\qquad}\forall u\in \overline B_{\alpha}(\mathcal{P}%
^{\fol}).
\]
Since $\text{dist}_{\theta}$ is given by a norm, by the triangle inequality it is enough to bound $\text{dist}_{\theta}(Lu,\mathcal{P}^{\fol})$ and $\text{dist}_{\theta}(Gu,\mathcal{P}^{\fol})$ separately. 

\smallskip
We first note that $Lv\in\mathcal{P}^{\fol}$ and $Gv\in\mathcal{P}^{\fol}$ if $v\in\mathcal{P}^{\fol}$. This follows from the fact that $M$ is compact, $\theta>0$, $\theta-b\geq0$ and $c>0$ and the Maximum Principle.  Indeed consider $Lv\colon M\to \R$ and denote by $D(Lv)(p)$ the derivative of $Lv$ at $p\in M$. Since $M$ is compact, then $Lv$ attains it minimum at some point $p\in M$. If we assume that  $Lv(p)<0$, then since we also know that $D(Lv)(p) = 0$  and $-\Delta(Lv)(p)\leqslant 0$, we have $(\theta-b)v(p) = -\Delta(Lv)(p)+\theta Lv(p) <0$, which is a contradiction. The same reasoning holds for $Gv$.

As in \cite[Lemma 5.1]{ClappFernandez2017},  the following inequality holds true
\begin{equation}\label{Eq:Vetois}
\Vert Lu \Vert_\theta \leq \frac{\theta - \mu}{\theta + \mu}\Vert u\Vert_\theta.
\end{equation}
For $u\in H_g^{1}(M)^{\fol}$ let $v\in\mathcal{P}^{\fol}$ be such that dist$_{\theta}%
(u,\mathcal{P}^{\fol})=\left\Vert u-v\right\Vert_\theta.$ Then, the previous paragraph, linearity of $L$ and inequality \eqref{Eq:Vetois} yield
\begin{linenomath}
\begin{align}
\begin{split}
\text{dist}_{\theta}(Lu,\mathcal{P}^{\fol})&\leq\left\Vert Lu-Lv\right\Vert_{\theta} \label{eq:distLu}\\
&\leq \tfrac{\theta-\mu}{\theta + \mu} \left\Vert u-v\right\Vert_\theta\\
&= \tfrac{\theta-\mu}{\theta + \mu}\,\text{dist}_{\theta}(u,\mathcal{P}^{\fol}).
\end{split}
\end{align}
\end{linenomath}

Now we bound $\text{dist}_{\theta}(Gu,\mathcal{P}^{\fol})$. For any function
$v\colon M \to \R$ we defined $v^\pm$ as $v^{+}=\max\{0,v\}$ and $v^{-}=\min\{0,v\}$. Then, since for $u\in H^1_g(M)^\fol$ we have $(Gu)^{+}\in \mathcal{P}^\fol$, by the definition of $\text{dist}_{\theta}$, using that $\left\langle
G(u)^+,G(u)^{-}\right\rangle_\theta=0$ and identity \eqref{Eq:Lu and Gu}, we get 
\begin{linenomath}
\begin{align*}
  \text{dist}_{\theta}(Gu,\mathcal{P}^{\fol})\left\Vert G(u)^{-}\right\Vert
_{\theta} \leqslant 
& \Vert G(u)-G(u)^{+} \Vert_\theta \,\Vert G(u)^{-}\Vert_\theta \\
=&\left\Vert G(u)^{-}\right\Vert_\theta^{2}\\
=&\left\langle
G(u),G(u)^{-}\right\rangle_\theta\notag\\
  =&\int_{M}c\left\vert u\right\vert ^{p-2}uG(u)^{-}dV_g.
\end{align*}
\end{linenomath}
To bound the integral above, recall that $c>0$ and $G(u)^-\leq 0$. Thus, we can apply 
H\"{o}lder's inequality, with $\frac{p-1}{p}+\frac{1}{p}=1$, to the functions 
$c^{(p-1)/p}\vert u^-\vert^{p-1}$  and 
$c^{1/p}\vert G(u)^-\vert$,
and then apply
(\ref{eq:distP})  and Sobolev's inequality, to get 
\begin{linenomath}
\begin{align*}
&\int_{M}c\left\vert u\right\vert ^{p-2}uG(u)^{-}dV_g \\
&= \int_{\{u>0\}} c \left\vert u\right\vert ^{p-2}uG(u)^{-}dV_g+\int_{\{u<0\}} c \left\vert u\right\vert ^{p-2}uG(u)^{-}dV_g\\
&\leq  \int_{\{u<0\}} c \left\vert u\right\vert ^{p-2}uG(u)^{-}dV_g\\
&= \int_{M}c\vert u^{-}\vert^{p-2}u^{-}%
G(u)^{-}\,dV_{g}\\
 &\leq \vert u^{-}\vert _{c,p}^{p
-1}\vert G(u)^{-}\vert _{c,p}\\
&\leq C^{p}\,\text{dist}_{\theta}(u,\mathcal{P}^{\fol})^{p-1}\Vert
G(u)^{-}\Vert_\theta.
\end{align*}
\end{linenomath}
Combining the inequalities above, we conclude that for all $ u\in H_{g}^{1}(M)^{\fol}$ with $\Vert Gu^{-}\Vert\neq 0$ it holds
\begin{equation} \label{eq:distGu}
\text{dist}_{\theta}(Gu,\mathcal{P}^{\fol})\leq C^{p}\,\text{dist}%
_{\theta}(u,\mathcal{P}^{\fol})^{p-1}. %
\end{equation}
This inequality is also true if $\Vert Gu^{-}\Vert= 0$, for, in this case $Gu=Gu^+$ and $\text{dist}_\theta(Gu,\mathcal{P}^\fol)=0$. So, inequality \eqref{eq:distGu} holds true for every $u\in H_g^1(M)^\fol.$

Fix $\nu\in\left( \tfrac{\theta-\mu}{\theta + \mu},1\right)  $ and let $\alpha_{2}>0$ be such
that $C^{p}\alpha_{2}^{p-2}  \leq\nu-\left(\tfrac{\theta-\mu}{\theta + \mu}\right)$. Then, for any
$\alpha\in(0,\alpha_{2}),$ by adding inequalities (\ref{eq:distLu}) and
(\ref{eq:distGu}) we obtain 
\[
\text{dist}_{\theta}(Lu+Gu,\mathcal{P}^{\fol})\leq\nu\,\text{dist}_{\theta}%
(u,\mathcal{P}^{\fol})<\text{dist}_\theta(u,\mathcal{P}^\fol),
\]
for all $u\in \overline B_{\alpha}(\mathcal{P}^{\fol})$. 
Therefore, $Lu+Gu\in\,B_{\alpha}(\mathcal{P}^{\fol})$ if $u\in \overline B_{\alpha
}(\mathcal{P}^{\fol})$.

\medskip
{\bf{Proof of Step 2:}}
Observe that we have not ruled out $\psi(u,t)\in \partial B_\alpha(\mathcal{P}^\fol)$. We now argue as in the proof of Lemma 2 in \cite{ClappWeth2005}. 

By contradiction, assume that there exists 
$u\in \overline B_\alpha(\mathcal{P}^\fol)$ such that 
$\psi(u,t)\in \partial B_\alpha(\mathcal{P}^\fol)$ for some $t \in (0, T(u))$. By Mazur's separation theorem, there exists a continuous linear functional $R \in (H^1_g(M)^\fol)^*$ and $\beta>0$ such that $R(\psi(u,t))-\beta$ and $R(v)> \beta$
for any $v\in B_\alpha(\mathcal P^\fol)$. It follows that 
\begin{linenomath}
\begin{align*}
    \tfrac{\partial}{\partial s}\big|_{s=t}R(\psi(u,s)) = 
    & R(-\nabla  J \psi(u,t)) \\
    & = R((L+G)\psi(u,t))-R(\psi(u,t)).
\end{align*}
\end{linenomath}
Since by hypothesis $\psi(u,t)\in \partial B_\alpha(\mathcal P^\fol)\subset \overline B_\alpha(\mathcal P^\fol)$, then by Step 1 above $(L+G)(\psi(u,t))\in B_\alpha(\mathcal P^\fol)$. Thus we obtain that
\begin{linenomath}
\[
    R((L+G)\psi(u,t))-R(\psi(u,t))= R((L+G)\psi(u,t))-\beta >0,
\]
\end{linenomath}
and consequently $\tfrac{\partial}{\partial s}\big|_{s=t}R(\psi(u,s)) >0$. Hence, there exists $\varepsilon>0$ such that $R(\psi(u,s))<\beta$ for $s\in (t-\varepsilon, t)$. Thus, $\psi(u,s) \not \in \overline B_\alpha(\mathcal{P}^\fol)$ for $s\in (t-\varepsilon, t)$. This contradicts Step 1. 

\medskip
Letting $\alpha_0:=\min\{\alpha_1,\alpha_2\}$ we conclude the proof of the lemma. \end{proof}

\bibliographystyle{siam}
\bibliography{Bibliography.bib}
\end{document}